\documentclass[VANCOUVER,STIX1COL]{WileyNJD-v2}

\articletype{Article Type}%

\received{26 April 2016}
\revised{6 June 2016}
\accepted{6 June 2016}

\usepackage{amsmath,amsfonts}
\usepackage{mathtools}
\usepackage{subfig}
\usepackage{tikz}
\usepackage{pgfplots}
\usepackage{xcolor}
\usepgfplotslibrary{groupplots,dateplot}
\usetikzlibrary{patterns,shapes.arrows}
\usepackage{siunitx}
\usepackage{nicefrac}

\begin{document}

\title{Why diffusion-based preconditioning of Richards equation works: spectral analysis and computational experiments at very large scale. 
\protect\thanks{This work has been partially supported by EU under the Horizon 2020 Project ``Energy oriented Centre of Excellence: toward exascale for energy'' (EoCoE--II), Project ID: 824158.}}

\author[1,2]{Daniele Bertaccini}

\author[2]{Pasqua D'Ambra}

\author[3,2]{Fabio Durastante*}

\author[4,2]{Salvatore Filippone}

\authormark{DANIELE BERTACCINI \textsc{et al}}

\address[1]{\orgdiv{Dipartimento di Matematica}, \orgname{Università di Roma ``Tor Vergata''}, \orgaddress{\state{Rome}, \country{Italy}}}

\address[2]{\orgdiv{Istituto per le Applicazioni del Calcolo ``Mauro Picone''}, \orgname{Consiglio Nazionale delle Ricerche}, \orgaddress{\state{Naples}, \country{Italy}}}

\address[3]{\orgdiv{Dipartimento di Matematica}, \orgname{Università di Pisa}, \orgaddress{\state{Pisa}, \country{Italy}}}

\address[4]{\orgdiv{Department of Civil and Computer Engineering}, \orgname{Università di Roma ``Tor Vergata''}, \orgaddress{\state{Rome}, \country{Italy}}}

\corres{*Fabio Durastante, Largo Bruno Pontecorvo 5, Pisa (PI) 56127, Italy. \email{fabio.durastante@unipi.it}}

\abstract[Summary]{
	We consider here a cell-centered finite difference approximation of the
	Richards  equation in three dimensions, averaging for interface
	values the hydraulic conductivity $K=K(p)$, a highly nonlinear
	function, by arithmetic, upstream and harmonic means. The
	nonlinearities in the equation can lead to changes in soil
	conductivity over several orders of magnitude and  
	discretizations with respect to space variables often produce stiff
	systems of differential equations. A fully implicit time
	discretization is provided by \emph{backward Euler} one-step formula;
	the resulting nonlinear algebraic system is solved by an
	inexact Newton Armijo-Goldstein algorithm, requiring the solution of
	a sequence of linear systems involving Jacobian matrices. 
	We prove some new results concerning the distribution 
	of the Jacobians eigenvalues and the explicit expression 
	of their entries.
	Moreover, we explore some connections between the saturation of the
	soil and the ill conditioning of the Jacobians. 
	The information on eigenvalues justifies the effectiveness of some preconditioner approaches which are widely used in the solution of Richards equation.
	We also propose a new software framework to experiment with scalable and robust preconditioners suitable for efficient parallel simulations at very large scales. Performance results on a literature test case show that our framework is very promising in the advance towards realistic simulations at extreme scale.}


\jnlcitation{\cname{%
\author{D'Ambra P.}, 
\author{D. Bertaccini}, 
\author{F. Durastante}, and 
\author{S. Filippone}} (\cyear{2022}), 
\ctitle{Why diffusion-based preconditioning of Richards equation works: spectral analysis and computational experiments at very large scale.}, \cjournal{ArXiv}, \cvol{\url{https://arxiv.org/abs/2112.05051}}.}

\maketitle

\section{Introduction}
\label{sec:introduction}

Groundwater flow in the unsaturated zone is a highly nonlinear
phenomenon that can be modeled by the Richards equation, and there is
a significant amount of research concerning
different formulations and algorithms for calculating the flow of
water through unsaturated porous media~\cite{Richards-review,Genuchten80,CELIA90,Miller-Kelley98,JonesWoodward01}.

The Richards equation is a time-dependent Partial Differential Equation 
(PDE) whose discretization leads to large nonlinear systems of
algebraic equations, which  often include coefficients showing 
large variations over different orders of magnitude. 
Typically, the variation in the 
coefficients is due to the use of a geostatistical model for
hydraulic conductivity, allowing changes of many orders of magnitude 
from one cell to the next (heterogeneity) as well as
correlations of values in each direction (statistical anisotropy).
High heterogeneity and anisotropy in the problem and
the presence of strong nonlinearities in the equation's coefficients
make the problem difficult to be solved numerically.

The main contribution of the present work is twofold. 
First, we investigate the spectral properties of the sequence of Jacobian matrices
arising in the solution of the non linear systems of algebraic equations by a quasi-Newton method; these
properties allow us to formulate a new theoretical justification for some preconditioning choices in 
solving Richards equation which are indeed current usual practices in the literature; to the best of our
knowledge, this is the first theoretical formulation to support such common practices. 
Then, on the basis of the theoretical indications, we discuss a scalable and efficient parallel solution 
of a modified inexact quasi-Newton method with Krylov solvers and Armijo-Goldstein's line search, 
i.e.  Newton-like algorithms where the Newton correction linear equations are solved by a Krylov subspace
method. To this aim, we interface the KINSOL package available from the Sundials
project~\cite{hindmarsh2005sundials} with the most recent version of libraries for solving sparse linear
systems by parallel Krylov solvers coupled with purely algebraic preconditioners~\cite{dambra2020b}, 
currently being extended in some EU-funded projects. 

The work is organized as follows. First, in
Section~\ref{sec:discretizing} we discuss the discretization of the 
Richards equation by means of cell-centered finite differences. Then, in
Section~\ref{sec:inexact-newton}, we summarize the inexact quasi-Newton
method as implemented in the KINSOL framework. 
In Section~\ref{sec:spectral-analysis} we propose an analysis of some
spectral properties of the sequence of Jacobian matrices produced by the
underlying Newton method.
In Section~\ref{sec:softwareframework}, we describe some of the main features of the PSCToolkit parallel
software framework,  which we use to implement the computational procedure described in this paper.

The information produced by the spectral analysis of the Jacobian matrices
is used in Section~\ref{sec:preconditioners} to devise a preconditioning
strategy for the Krylov subspace method and to present the preconditioners
used in this study. Some numerical examples highlighting the computational
performance of the proposed methods are then presented  in 
Section~\ref{sec:numerical_experiment}. Finally, in
Section~\ref{sec:conclusions} we draw our conclusion and briefly
discuss future extensions.

\section{Formulating the discrete problem}
\label{sec:discretizing}

Flow in the vadose zone, which is also known as the
\emph{unsaturated} zone, has rather delicate aspects such as the parameters
that control the flow, which depend on the saturation of the media,
leading to the nonlinear problem described by the Richards equation. 
The flow  can be expressed as a combination of Darcy's law and the
principle of mass  conservation by
\[\frac{\partial \left(\rho \,\phi s(p) \right)}{\partial t} + \nabla \cdot q = 0,\]
in which $s(p)$ is the saturation at pressure head $p$ of a fluid
with density $\rho$ in terrain with porosity $\phi$, and $q$ is
the volumetric water flux. By using Darcy's  law as
\[q = - K(p) \left( \nabla p + c \hat{z} \right),\]
for $K(p)$ the hydraulic conductivity, and $c$ the cosine of the
angle between the downward $z$ axis $\hat{z}$ and the direction of
the gravity force, we obtain that the overall equation has a
diffusion as well as an advection term, the latter being  related to
gravity. For the sake of simplicity, we  consider cases in which $c
= 1$, i.e. the advection acts only in the $z$
direction. There are two different forms of the Richards equation that
differ in how they deal with the nonlinearity in the time derivative. The most popular form, 
which  is the one considered here for a fluid that will always be water, permits to 
express the general equation as
\begin{equation}\label{eq:richards}
\rho \,\phi\frac{\partial s(p) }{\partial t} - \nabla\cdot
K(p)\nabla p - \frac{\partial K(p)}{\partial z}  = f,
\end{equation}
where $p(t)$ is the pressure head at time $t$, $s(p)$ is water saturation
at pressure head $p$, $\rho$ is water density, $\phi$ is the porosity of
the medium, $K(p)$ the hydraulic conductivity, $f$ represents any water
source terms and $z$ is elevation. The equation is then completed by
adding boundary and initial conditions.
This formulation of the Richards equation is called the \emph{mixed
	form}~\cite{Richards-review,CELIA90} because the equation is parameterized in $p$ (pressure head)
but the time derivative is in terms of $s$ (water saturation). Another formulation of the 
Richards equation is known as the
head-based form, and is popular because the time derivative is written
explicitly  in terms of  the pressure head $p$~\cite{Genuchten80,CELIA90}.

It is important to note that in  unsaturated flow  both water
content, $\phi s(p)$, and hydraulic conductivity, $K(p)$, are functions of
the pressure head  $p$. There are several empirical relations used to relate
these parameters, including, e.g., the
Brooks-Corey~\cite{brooks1964hydraulic} model and the Van Genuchten
model~\cite{Genuchten80}. The Van Genuchten model is slightly more
popular because it containts no discontinuities in the functions, unlike 
the Brooks-Corey model, and therefore avoids the risk of losing the
uniqueness of the solution of the semidiscretized equation in
space by the well-known Peano Theorem. A version of the Van
Genuchten-Richards equation in mixed form model by Celia et al.~\cite{CELIA90} reported the following choices:
\begin{equation}\label{eq:s(p)}
s(p)=\frac{\alpha (s_s-s_r)}{\alpha+|p|^{\beta}}+s_r,
\end{equation}
and
\begin{equation}\label{eq:K(p)}
K(p)=K_s \frac{a}{a+|p|^{\gamma}}
\end{equation}
where the $\alpha,\beta,\gamma$, and $a$ are
fitting parameters that are often assumed to be constant in the
media; $s_r$ and $s_s$ are the residual and saturated moisture
contents, and $K_s$ is the saturated hydraulic conductivity.

Small changes in the pressure head can change the hydraulic conductivity
by several orders of magnitude, and $K(p)$ is a highly nonlinear
function. The water content curve is also highly nonlinear as
saturation can change drastically over a small range of pressure
head values. It should be noted that these functions are only valid
when the pressure head is negative; that is, the media is
semi-saturated; when the media is fully saturated, $K = K_s$, 
$s(p)$ is equal to the porosity, and the Richards equation
reduces to the Darcy linear equation; see the example in
Figure~\ref{fig:VanGenuchtenModel}.

\begin{figure}[htbp]
	\centering
	\input{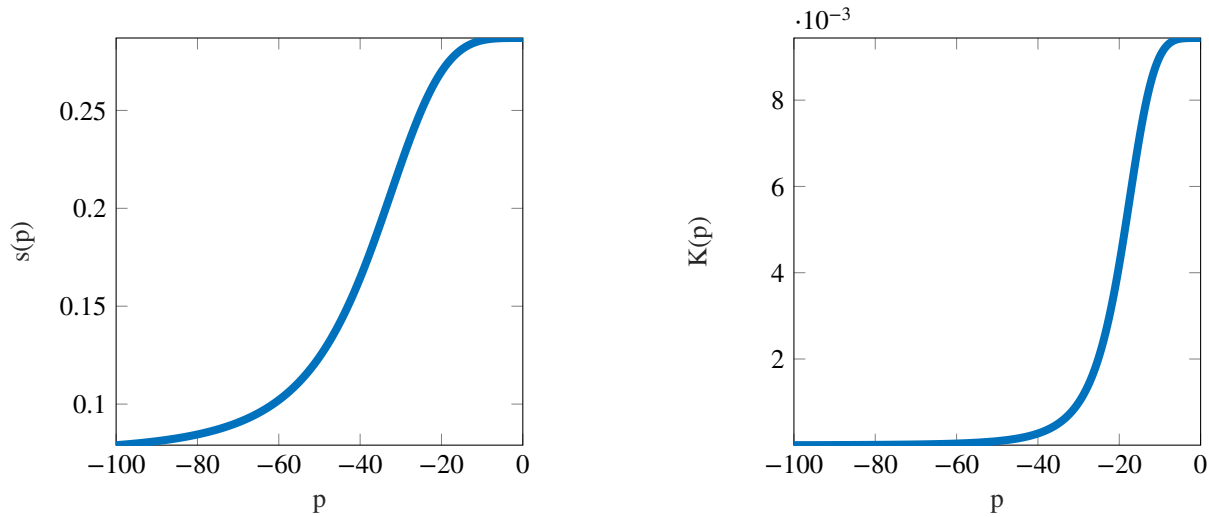}
	\caption{Hydraulic conductivity $K(p)$ and volumetric water content $S(p)$ for parameters $\alpha = \texttt{1.611e+6}$, $\beta  = 3.96$ , $\gamma = 4.74$,  $a     = \texttt{1.175e+6}$, $S_s    = 0.287$, $S_r    = 0.075$,  $K_s    = 0.00944$ \si{cm/s}.}
	\label{fig:VanGenuchtenModel}
\end{figure}

Typical uses of the Richards equation are to simulate infiltration
experiments in the laboratory and in the field. These
begin with initially dry soil,  then water is added to
the top of the core sample (or ground surface).
These experiments  require the  simulation of  fast changes in pressure head and saturation over the most
nonlinear part of the Van Genuchten curves.

\subsection{Cell-centered finite difference discretization}\label{sec:finite_difference_discretization}

At the beginning of an  infiltration experiment,
the pressure head $p$ can be close to discontinuous. These large
changes are also reflected in the nonlinear terms $K$ and
$s(p)$. Taking into account the effect of the initial conditions, the time step 
can be severely limited if an inappropriate time discretization is
chosen. Hydrogeologists are often interested in following the process
until a steady state is achieved, which may take a long time; therefore,
large time steps should be used to avoid excessively expensive 
simulations. The presence of stiffness and the desire to use long time steps,
lead to implicit time integration schemes. 
In this paper we consider an implicit backward Euler numerical scheme;  higher-order 
implicit methods are not considered because 
the uncertainty associated with the fitting parameters in the Van
Genuchten models and the possible low smoothness of $s$ (i.e., the
max order of differentiability) can severely reduce the effective order of
a (theoretically) high order numerical method.

Successful discretizations in space for Richards equation are based
on finite differences~\cite{CELIA90,JonesWoodward01}; using finite
elements may require mass lumping in order to recover possible large mass
balance errors and undershoot errors ahead of the infiltration
front~\cite{CELIA90}.

We consider here a discretization of the Richards equation on a cell-centered
finite difference tensor mesh on a parallelepiped  discretized
with $\mathbf{N} = (N_x,N_y,N_z)$ nodes, in such a way that the most
external nodes are on the physical boundary of the domain. We then denote 
the cell centers as 
$\{x_{i,j,k} = (i h_x, j h_y, k h_z)\}_{i,j,k=0}^{N-1}$, 
for
$\mathbf{h} = (h_x,h_y,h_z) = (L_x,L_y,L)/(\mathbf{N}-1)$, 
and with the interfaces located at midpoints between adjacent nodes.
The time direction is then discretized by considering $N_t$ uniform time steps, i.e., the grid $\{t_l = l \Delta\,t \}_{l=0}^{N_t-1}$ for $\Delta\,t=1/(N_t-1)$; the approximation on the cell $(i,j,k)$ at time step $l$ of the pressure head $p$ is denoted as $p_{i,j,k}^{(l)}$.

With this choice, the discretization of the mixed form of the Richards
equation in~\eqref{eq:richards} on the internal nodes of the grid for $l\geq 1$ is written, for $i,j,k=1,\ldots,\mathbf{N}-2$, as
\begin{equation}\label{eq:dynamic_of_the_problem}
\begin{split}
\boldsymbol{\Phi}(p_{i,j,k}^{(l)}) = & \frac{\rho  \phi }{\Delta t} \left(s\left(p_{i,j,k}^{(l)}\right)-s\left(p_{i,j,k}^{(l-1)}\right)\right)  + q_{i+1/2,j,k}^{(l)} - q_{i-1/2,j,k}^{(l)}  + q_{i,j+1/2,k}^{(l)} - q_{i,j-1/2,k}^{(l)} + q_{i,j,k+1/2}^{(l)} - q_{i,j,k-1/2}^{(l)} + f_{i,j,k} \equiv 0,
\end{split}
\end{equation}
with
\begin{eqnarray*}
	q_{i+1/2,j,k}^{(l)} = & -\prescript{\text{\scriptsize AV}}{}{K_{i+1,i}^{(l)}} \left( \frac{p^{(l)}_{i+1,j,k}-p^{(l)}_{i,j,k} }{h_x^2} \right), \qquad
	q_{i-1/2,j,k}^{(l)} = & -\prescript{\text{\scriptsize AV}}{}{K_{i-1,i}^{(l)}} \left( \frac{p^{(l)}_{i,j,k}-p^{(l)}_{i-1,j,k} }{h_x^2} \right), \\
	q_{i,j+1/2,k}^{(l)} = & -\prescript{\text{\scriptsize AV}}{}{K_{j+1,j}^{(l)}} \left( \frac{p^{(l)}_{i,j+1,k}-p^{(l)}_{i,j,k} }{h_y^2} \right), \qquad
	q_{i,j-1/2,k}^{(l)} = & -\prescript{\text{\scriptsize AV}}{}{K_{j-1,j}^{(l)}} \left( \frac{p^{(l)}_{i,j,k}-p^{(l)}_{i,j-1,k} }{h_y^2} \right),\\
	q_{i,j,k+1/2}^{(l)} = & -\prescript{\text{\scriptsize AV}}{}{K_{k+1,k}^{(l)}} \left( \frac{p^{(l)}_{i,j,k+1}-p^{(l)}_{i,j,k} }{h_z^2} \right) - \frac{K(p_{i,j,k+1})}{2h_z}, \qquad q_{i,j,k-1/2}^{(l)} = & -\prescript{\text{\scriptsize AV}}{}{K_{k-1,k}^{(l)}} \left( \frac{p^{(l)}_{i,j,k}-p^{(l)}_{i,j,k-1} }{h_z^2} \right) - \frac{K(p_{i,j,k-1})}{2h_z},
\end{eqnarray*}
in which the generic term $\prescript{\text{\scriptsize AV}}{}{K_{L,U}}^{(l)}$ represents an average at the interfaces of the function $K(p)$ in~\eqref{eq:K(p)}. The selection of the form of the average term best suited to more realistic simulations does depend on the problem and is still an open issue~\cite{doi:10.1029/2008WR007654,doi:10.1061/(ASCE)1084-0699(2006)11:6(526)}. In particular, the choice is dependent on the type of fluid infiltration one needs to deal with. If we denote by $K_U$ and $K_L$ the two values of the function $K$ on the opposite sides of the interface, e.g., $K_U = K(p^{(l)}_{i,j+1,k})$ and $K_L = K(p^{(l)}_{i,j,k})$, the most frequently used formulations are the arithmetic mean~\cite{CELIA90}, i.e., $\prescript{\text{\scriptsize ARIT}}{}{K_{L,U}}^{(l)} = (K_U+K_L)/2$, the geometric mean~\cite{doi:10.1029/WR015i001p00181}, i.e., $\prescript{GEOM}{}{K^{(l)}} = \sqrt{K_U K_L}$, the upstream-weighted mean~\cite{doi:10.1029/92WR02875}, i.e.,
\begin{equation}\label{eq:upstream_mean}
\prescript{\text{UP}}{}{K^{(l)}} = \begin{cases}
K_U, & p_U - p_L \geq 0,\\
K_L, & p_U - p_L < 0,\\
\end{cases}
\end{equation}
and the integral mean~\cite{Srivastava1995}, i.e.,
\begin{equation*}
\prescript{\text{INT}}{}{K^{(l)}} = \begin{cases}
\frac{1}{p_L - p_U} \int_{p_L}^{p_U} K(\psi) d\psi, & p_L \neq p_U, \\
p_U, & \text{ otherwise}.
\end{cases}
\end{equation*}
It is also possible to employ  a combination of the above, using two different means, one for the horizontal
and another for the vertical direction, or to use an algorithm that computes the internodal conductivity by
using  different approaches depending on the terrain and the value of the pressure
head~\cite{doi:10.1029/2008WR007654}.

\section{Applying the inexact Newton method}
\label{sec:inexact-newton}

We apply at each time step a quasi-Newton method for the solution of the
nonlinear system of equations~\eqref{eq:dynamic_of_the_problem}
as implemented in the KINSOL parallel library~\cite{hindmarsh2005sundials}. 
Let $\mathbf{p}_r$ be the current iterate of pressure head, for each node of the computational mesh and for each time step. A Newton method computes an increment $\mathbf{d}_r$ as the solution of the following equation: 
\begin{equation}\label{eq:linearsystemtosolve}
J(\mathbf{p}_r) \mathbf{d}_r = -\boldsymbol{\Phi}(\mathbf{p}_r),
\end{equation}
where $J(\mathbf{p}_r)$ is the Jacobian matrix of $\boldsymbol{\Phi}$. 
Specifically, we consider an
inexact Newton solver in which we update the Jacobian matrix as
infrequently as possible. This means that a first Jacobian is
computed in the initialization phase, i.e. in the first step $r=0$
of the algorithm; then, a new one is built if and only if at least one of
the following conditions are met
\begin{equation}\label{eq:jacobian_update_condition}
\left\lbrace\begin{array}{l}
r \equiv 0 \,(\text{ mod }10), \\
\| \lambda d_{r-1} \|_{D_u,\infty} > 1.5, \\
\|\lambda d_r \|_{D_u,\infty} < \varepsilon_{\text{machine}}^{2/3} \\
\text{ The linear solver failed to converge with the}\\
\quad \text{ previous Jacobian (backtracking), or}\\
\text{ The line search failed with outdated Jacobian information,}\\
\end{array}
\right.
\end{equation}
where we are using the scaled norm $\| \cdot \|_{D_u,\infty} = \|
D_u \cdot \|_\infty$ for a diagonal matrix $D_u$ such that the
entries of the scaled vector are almost of the same magnitude when
$\mathbf{p}_r$ is close to a solution. Similarly, $\| \cdot
\|_{D_F,\infty} = \| D_F \cdot \|_\infty$ for a diagonal matrix
$D_F$ when we are far from the solution. The value of the step
length $\lambda$ is computed via the Armijo-Goldstein line search
strategy, i.e., $\lambda$ is chosen to guarantee a sufficient
decrease in $\boldsymbol{\Phi}$ with respect to the step length as well
as a minimum step length to the initial rate of decrease of~$\mathbf{\Phi}$.

The Jacobian matrix $J = J_{\mathbf{\Phi}}$ can then be recovered by
direct computation from~\eqref{eq:dynamic_of_the_problem} 
using finite-difference approximations to the derivatives of the constitutive equations
in~\eqref{eq:s(p)}, and~\eqref{eq:K(p)} given by
\begin{equation}\label{eq:derivativeofcoefficients}
s'(p) = -\frac{\alpha  \beta  \left| p\right| ^{\beta -1} \text{sgn}(p) \left(S_s-S_r\right)}{\left(\alpha +\left| p\right| ^{\beta }\right)^2}, \text{ and } K'(p) = -\frac{a \gamma  k_s \left| p\right| ^{\gamma -1} \text{sgn}(p)}{\left(a+\left| p\right| ^{\gamma }\right)^2}.
\end{equation}
At the core of the parallel procedure we tackle  the
solution of the (right preconditioned) linear system
\begin{equation}\label{eq:precsystem}
J M^{-1} (M \mathbf{d}_r) = -\boldsymbol{\Phi}(\mathbf{p}_r),
\end{equation}
where $J$ could be either a freshly computed Jacobian for the vector
$\mathbf{p}_r$, or the Jacobian coming from a previous step.
The linear iterative solver for the Newton equations should handle  the
tradeoff between using a tolerance that is large enough to avoid
oversolving, and reducing the number of iterations to attain the 
prescribed convergence.
In the following we focus on a choice for the preconditioner $M$
in~\eqref{eq:precsystem} which allows to balance accuracy and efficiency
on parallel distributed-memory architectures when very large scale
simulations have to be carried out. 

In Section~\ref{sec:spectral-analysis} we  investigate the asymptotic spectral properties of
the sequence of the Jacobian matrices.  This information will then be used to devise
an \emph{asymptotically spectrally equivalent} 
	 sequence of symmetric and positive definite matrices. Then, to approximate the action of the inverses of
	the approximating sequence, we will employ some preconditioners from the PSCToolkit framework described in
	Section~\ref{sec:softwareframework}. The aim to experiment the functionalities of PSCToolkit for building
	and apply preconditioners for the linear systems arising in the quasi-Newton procedure by KINSOL, we
	developed some KINSOL modules that enable it to use the solvers and
    preconditioners from PSCToolkit inside its Newton-based non-linear procedure. These interfaces are written
    in C from the KINSOL library and use the C/Fortran2003 interfaces from the PSCToolkit libraries, 
    guaranteeing a full interoperability of the data structures, i.e., they do not require producing any
    auxiliary copy of KINSOL objects for translation
into PSCToolkit objects; everything can be manipulated from KINSOL directly into the native formats for PSCToolkit. The details about the
implementation of the relevant APIs, and the operators made available
by the interfacing are described in the documentation for the
interface that can be downloaded
from~\url{https://github.com/Cirdans-Home/kinsol-psblas}. This interfacing had the twofold aim of extending
PSCToolkit to handle non-linear algebraic equations, as well as to extend the KINSOL library with new
methods for solving sparse linear systems on high-end supercomputers.

\section{Spectral analysis of the Jacobian sequence}
\label{sec:spectral-analysis}

In general, the construction of the preconditioner $M$
in~\eqref{eq:precsystem} depends  on the
choice of the average for the discretization used
in~\eqref{eq:dynamic_of_the_problem}. To formulate a proposal for $M$, 
we divide the discussion in two steps: first we look for a suitable 
matrix $M$ to precondition the Jacobian matrix $J$ associated to the
different averages;  then we discuss how we can efficiently setup and apply $M^{-1}$ inside the Krylov subspace method on a high-end parallel
computer.

\subsection{Tools for the spectral analysis}
Our idea to compute $M$ starts by investigating the 
distribution of the eigenvalues,
$\{\lambda_j(J_{\mathbf{N}})\}_{\mathbf{j}=\mathbf{1}}^{\mathbf{N} = (N_x,N_y,N_z)}$, for the
Jacobian matrix $J_{\mathbf{N}}$ of size $N = N_x N_y N_z$.
Specifically, we look for a measurable function $f : D \subset
\mathbb{R}^{k} \rightarrow \mathbb{C}$ to associate to the sequence
$\{J_{\mathbf{N}} \}_{\mathbf{N}}$ for which we can prove the
following {\it asymptotic} relation
\[ \lim_{N \to\infty}\frac1{N}\sum_{i=1}^{N}F(\lambda_i(J_\mathbf{N}))=\frac1{\mu_k(D)}\int_DF(f(\mathbf{x})){\rm d}\mathbf{x},\qquad\forall\,F\in C_c(\mathbb C), \]
where $\mu_k(\cdot)$ is the Lebesgue measure on $\mathbb{R}^k$, $0 < \mu_k(D) < \infty$, and
$C_{c}(\mathbb{C})$ is the space of continuous functions with compact support. Despite the apparently
technical form of the previous expression, we can easily summarize  the information contained in $f$. If we
assume that $N$ is large enough, then the eigenvalues of the matrix $J_{\mathbf{N}}$, except possibly for
$o(N)$ outliers, are approximately equal to the samples of $f$ over a uniform grid in $D$, i.e., the function
$f$, that we will call {\it symbol} of the sequence of matrices $\{J_{\mathbf{N}}\}_{\mathbf{N}}$, provides an
accurate description of their spectrum asymptotically. %

In order to achieve this result, we need to introduce some preliminary tools. To
simplify notation, let us start by
focusing on the one dimensional problem to select only
an expression for $\prescript{\text{\scriptsize AV}}{}{K_{L,U}}$
related to the flux oriented toward the $z$ axis. For these cases, we
are interested in the overall behavior for both the eigenvalues and
the singular values of the Jacobian matrices. Formally, we are
interested in the computation of the so-called \textit{singular} and
\textit{spectral} value symbol for the sequence of the Jacobians.

\begin{definition}
	Let $\{A_N\}_N$ be a sequence of matrices and let $f:D\subset\mathbb
	R^k\to\mathbb C$ be a measurable function defined on a set $D$ with
	$0<\mu_k(D)<\infty$.
	\begin{itemize}
		\item $\{A_N\}_N$ has a singular value distribution described by $f$, and we write $\{A_N\}_N\sim_\sigma f$, if
		\[ \lim_{N\to\infty}\frac1{N}\sum_{i=1}^{N}F(\sigma_i(A_N))=\frac1{\mu_k(D)}\int_DF(|f(\mathbf{x})|){\rm d}\mathbf{x},\qquad\forall\,F\in C_c(\mathbb R). \]
		In this case, $f$ is called the {\it singular value symbol} of $\{A_N\}_N$.
		\item $\{A_N\}_N$ has a spectral (or eigenvalue) distribution described by $f$, and we write $\{A_N\}_N\sim_\lambda f$, if
		\[ \lim_{N\to\infty}\frac1{N}\sum_{i=1}^{N}F(\lambda_i(A_N))=\frac1{\mu_k(D)}\int_DF(f(\mathbf{x})){\rm d}\mathbf{x},\qquad\forall\,F\in C_c(\mathbb C). \]
		In this case, $f$ is called the {\it spectral (or eigenvalue) symbol} of $\{A_N\}_N$.
	\end{itemize}
	If $\{A_N\}_N$ has both a singular value and a spectral distribution
	described by $f$, we write $\{A_N\}_N\sim_{\sigma,\lambda}f$.
\end{definition}
Moreover, we refer to a sequence of matrices $\{Z_N\}_N$ such that
$\{Z_N\}_N\sim_\sigma0$ as a zero-distributed sequence. Examples of
sequences for which we can easily compute such symbols are the $n$th
diagonal sampling matrix generated by $a:[0,1]\to\mathbb C$ and
$N\in\mathbb N$, that is the $N\times N$ diagonal matrix given by
\begin{equation*}
D_N(a)=\mathop{\rm diag}_{i=1,\ldots,N}a\Bigl(\frac iN\Bigr),
\end{equation*}
and the Toeplitz sequences, i.e., for a $N\in\mathbb N$ and $f:[-\pi,\pi]\to\mathbb C$ a $L^1([-\pi,\pi])$ 
function, the $N\times N$ matrix
\[ T_N(f)=[f_{i-j}]_{i,j=1}^N, \]
where the numbers $f_k$ are the Fourier coefficients of $f$,
\[ f_k=\frac1{2\pi}\int_{-\pi}^\pi f(\theta){\rm e}^{-{\rm i}\,k\theta}{\rm d}\theta,\qquad k\in\mathbb Z. \]
For these sequences, if $f\in L^1([-\pi,\pi])$ then $\{T_N(f)\}_N\sim_\sigma f$, while if $f\in
L^1([-\pi,\pi])$ and $f$ is real then $\{T_N(f)\}_N\sim_\lambda f$. A similar relationship holds for the case
of multilevel Toeplitz matrices.

To compute the asymptotic spectral/singular value distribution of
more general matrix sequences we need to expand our set of tools, and 
for this task we use the Generalized Locally Toeplitz (GLT) sequences~\cite{Garoni2017,Garoni2018}. 
GLT sequences are  are sequences of matrices equipped with a measurable function
$\kappa:[0,1]\times[-\pi,\pi]\to\mathbb C$ called {\it symbol}; we will use the notation 
$\{A_N\}_N\sim_{\rm GLT}\kappa$ to  indicate that $\{A_N\}_N$ is a GLT sequence with symbol $\kappa$. 
We can characterize the sequences by the following list of properties.
\begin{description}
	\item[\textbf{GLT\,1.}] If $\{A_N\}_N\sim_{\rm GLT}\kappa$ then $\{A_N\}_N\sim_\sigma\kappa$. If $\{A_N\}_N\sim_{\rm GLT}\kappa$ and the matrices $A_N$ are Hermitian then $\{A_N\}_N\sim_\lambda\kappa$.
	\item[\textbf{GLT\,2.}] If $\{A_N\}_N\sim_{\rm GLT}\kappa$ and $A_N=X_N+Y_N$, where
	\begin{itemize}
		\item every $X_N$ is Hermitian,
		\item $\|X_N\|,\,\|Y_N\|\le C$ for some constant $C$ independent of $n$,
		\item $n^{-1}\|Y_N\|_1\to0$,
	\end{itemize}
	then $\{A_N\}_N\sim_\lambda\kappa$.
	\item[\textbf{GLT\,3.}] We have
	\begin{itemize}
		\item $\{T_N(f)\}_N\sim_{\rm GLT}\kappa(x,\theta)=f(\theta)$ if $f\in L^1([-\pi,\pi])$,
		\item $\{D_N(a)\}_N\sim_{\rm GLT}\kappa(x,\theta)=a(x)$ if $a:[0,1]\to\mathbb C$ is Riemann-integrable, %
		\item $\{Z_N\}_N\sim_{\rm GLT}\kappa(x,\theta)=0$ if and only if $\{Z_N\}_N\sim_\sigma0$.
	\end{itemize}
	\item[\textbf{GLT\,4.}] If $\{A_N\}_N\sim_{\rm GLT}\kappa$ and $\{B_N\}_N\sim_{\rm GLT}\xi$ then
	\begin{itemize}
		\item $\{A_N^*\}_N\sim_{\rm GLT}\overline\kappa$,
		\item $\{\alpha A_N+\beta B_N\}_N\sim_{\rm GLT}\alpha\kappa+\beta\xi$ for all $\alpha,\beta\in\mathbb C$,
		\item $\{A_NB_N\}_N\sim_{\rm GLT}\kappa\xi$.
	\end{itemize}
	\item[\textbf{GLT\,5.}] If $\{A_N\}_N\sim_{\rm GLT}\kappa$ and $\kappa\ne0$ a.e.\ then $\{A_N^\dag\}_N\sim_{\rm GLT}\kappa^{-1}$.
\end{description}
We call \emph{unilevel} sequences those for which the dimension is characterized by a single index $N$, we call $d$-level those in 
which the dimension is characterized by a multi-index $\mathbf{N} \in \mathbb{N}^d$. All definitions can be transparently generalized to this context.

\begin{remark}
	In all the following analysis, to simplify the notation, we remove
	from the vector $\mathbf{p}^{(k,l+1)}$ the index $k$ denoting the
	dependence on the iterate of the Newton method. The spectral
	analysis uses only the entries $p_i^{(l+1)}$ and does not
	depend on the step at which such a vector has been obtained.
\end{remark}

\subsection{Arithmetic average}\label{sec:arithmetic_average}
The simplest choice for the $\prescript{\text{\scriptsize
		AV}}{}{K_{L,U}}$ is given by the arithmetic mean of the values of
$K$ in~\eqref{eq:K(p)} at the two sides of the interface, i.e.,
\begin{equation*}
\begin{split}
\prescript{\text{\tiny ARIT}}{}{K_{i+1,i}^{(l+1)}} = & \frac{1}{2} \left(K(p^{(l+1)}_i) + K(p_{i+1}^{(l+1)})\right), \\ \prescript{\text{\tiny ARIT}}{}{K_{i-1,i}^{(l+1)}} = &  \frac{1}{2} \left(K(p^{(l+1)}_i) + K(p_{i-1}^{(l+1)})\right).
\end{split}
\end{equation*}
With this choice, equation~\eqref{eq:dynamic_of_the_problem} becomes
\begin{equation*}
\begin{split}
\prescript{{\text{\tiny ARIT}}}{}{\Phi}(p^{(l+1)}_i) = & \frac{S\left(p_i^{(l+1)}\right)-S\left(p_i^{(l)}\right)}{\Delta t} \\ & -\frac{1}{2{h_z}^2}\left[ \left(p_{i+1}^{(l+1)}-p_i^{(l+1)}\right) \left(K\left[p_{i+1}^{(l+1)}\right]+K\left[p_i^{(l+1)}\right]\right)\right.\\&\left.- \left(p_i^{(l+1)}-p_{i-1}^{(l+1)}\right) \left(K\left[p_{i-1}^{(l+1)}\right]+K\left[p_i^{(l+1)}\right]\right)\right] \\ & -\frac{1}{2 h_z}\left( K\left[p_{i+1}^{(l+1)}\right]-K\left[p_{i-1}^{(l+1)}\right] \right)
\end{split}
\end{equation*}
Then, the Jacobian matrix is, as in all the other one dimensional cases, a tridiagonal matrix with entries
\begin{equation}\label{eq:jacobian_arithmetic_mean}
\prescript{{\text{\tiny ARIT}}}{}{J_N} = \operatorname{tridiag}(\eta_i,\zeta_i,\xi_i) = \begin{bmatrix}
\zeta_1 & \xi_1 \\
\eta_1 & \ddots & \ddots \\
& \ddots &  \ddots & \xi_{N-3} \\
& & \eta_{N-3} & \zeta_{N-2}
\end{bmatrix}
\end{equation}
where,
\begingroup
\allowdisplaybreaks
\begin{align*}
\zeta_i = & \frac{\Delta t K\left[p_{i-1}^{(l+1)}\right]+\Delta t K\left[p_{i+1}^{(l+1)}\right]+2 \Delta t K\left[p_i^{(l+1)}\right]+2 h_z^2 S'\left(p_i^{(l+1)}\right)}{2 \Delta t h_z^2}\\ & -\frac{p_{i-1}^{(l+1)} K'\left(p_i^{(l+1)}\right)}{2 h_z^2}+\frac{p_i^{(l+1)} K'\left(p_i^{(l+1)}\right)}{h_z^2}-\frac{p_{i+1}^{(l+1)} K'\left(p_i^{(l+1)}\right)}{2 h_z^2},\\
\xi_i = & -\frac{K\left[p_{i+1}^{(l+1)}\right]+K\left[p_i^{(l+1)}\right]+h_z K'\left(p_{i+1}^{(l+1)}\right)}{2 h_z^2}-\frac{p_{i+1}^{(l+1)} K'\left(p_{i+1}^{(l+1)}\right)}{2 h_z^2}\\ & +\frac{p_i^{(l+1)} K'\left(p_{i+1}^{(l+1)}\right)}{2 h_z^2},\\
\eta_i = & -\frac{K\left[p_{i-1}^{(l+1)}\right]+K\left[p_i^{(l+1)}\right]-h_z K'\left(p_{i-1}^{(l+1)}\right)}{2 h_z^2}-\frac{p_{i-1}^{(l+1)} K'\left(p_{i-1}^{(l+1)}\right)}{2 h_z^2}\\ & +\frac{p_i^{(l+1)} K'\left(p_{i-1}^{(l+1)}\right)}{2 h_z^2}.
\end{align*}
\endgroup
To perform the spectral analysis we first rewrite the
Jacobian matrices as a sum of matrices that are easier  to investigate, 
then we use then the $*$-algebra (\textbf{GLT4}) and perturbation
techniques (\textbf{GLT2}) to obtain the spectral information we look for.
In particular, to analyze the Jacobian
in~\eqref{eq:jacobian_arithmetic_mean} we
separate it into three parts, one for the time-stepping, one
taking into account the Darcian diffusion, and the last one related to the
transport along the $z$-axis, i.e.,
\begin{equation*}
\{ \prescript{{\text{\tiny ARIT}}}{}{J_N} \}_N = \{ \prescript{{\text{\tiny ARIT}}}{}{D_N} \}_{N} + \{ \prescript{{\text{\tiny ARIT}}}{}{L_N} \}_N + \{ \prescript{{\text{\tiny ARIT}}}{}{T_N} \},
\end{equation*}
where $\{ \prescript{{\text{\tiny ARIT}}}{}{D_N} \}_{N}$ is the
scaled diagonal matrix sampling $s'(p)$ on the function $p$
approximated at the $(i+1)$th time step, i.e.,
\begin{equation}\label{eq:arithmeticdiagonalsampling}
\prescript{{\text{\tiny ARIT}}}{}{D_N} = \frac{1}{\Delta t} \operatorname{diag}(s'[p_i^{(l+1)}])_{i=1}^{N-2}.
\end{equation}
We can express the Darcian diffusion part as
\begingroup
\allowdisplaybreaks
\begin{align*}
\prescript{{\text{\tiny ARIT}}}{}{L_N} = & \frac{1}{2 h_z^2}\operatorname{tridiag}\left(-K\left[p_{i-1}^{(l+1)}\right],2 K\left[p_i^{(l+1)}\right],-K\left[p_{i+1}^{(l+1)}\right]\right)_{i=1}^{N-2} \\
& +  \frac{1}{2 h_z^2}\operatorname{tridiag}\left( 0,K\left[p_{i-1}^{(l+1)}\right],-K\left[p_i^{(l+1)}\right]\right)_{i=1}^{N-2} \\ & +  \frac{1}{2 h_z^2}\operatorname{tridiag}\left( -K\left[p_i^{(l+1)}\right],K\left[p_{i+1}^{(l+1)}\right],0\right)_{i=1}^{N-2} \\
& + \frac{1}{2 h_z^2}\operatorname{tridiag}\left(-p_{i-1}^{(l+1)} K'\left[p_{i-1}^{(l+1)}\right],2p_i^{(l+1)} K'\left[p_i^{(l+1)}\right], -p_{i+1}^{(l+1)} K'\left[p_{i+1}^{(l+1)}\right] \right)_{i=1}^{N-2} \\
& + \frac{1}{2 h_z^2}\operatorname{tridiag}\left(0,-p_{i-1}^{(l+1)} K'\left[p_i^{(l+1)}\right],p_i^{(l+1)} K'\left[p_{i+1}^{(l+1)}\right] \right)_{i=1}^{N-2} \\ & + \frac{1}{2 h_z^2}\operatorname{tridiag}\left( p_i^{(l+1)} K'\left[p_{i-1}^{(l+1)}\right],-p_{i+1}^{(l+1)} K'\left[p_i^{(l+1)}\right] ,0\right)_{i=1}^{N-2},
\end{align*}
\endgroup
while the remaining part is given by
\begin{equation}\label{eq:arithmetictransportterm}
\begin{split}
\prescript{{\text{\tiny ARIT}}}{}{T_N} = & \frac{1}{2} \operatorname{tridiag}\left(-K'\left[p_{i-1}^{(l+1)}\right],0,K'\left[p_{i+1}^{(l+1)}\right]\right)_{i=1}^{N-2}.
\end{split}
\end{equation}

\begin{lemma}\label{lem:symboltime}
	The matrix sequence $\{h_z^2 \prescript{{\text{\tiny ARIT}}}{}{D_N}\}_N$ has GLT symbol $C s'(p(\psi(x)))$, for $s'$ in~\eqref{eq:derivativeofcoefficients}, for $\psi(x)$ the function mapping the $[0,1]$ interval to the domain of definition for the $z$ variable, and $C = \lim_{N,N_t \rightarrow +\infty} h_z^2/\Delta t$.
\end{lemma}

\begin{proof}
	Observe that the ratio above is such that
	\begin{equation*}
	\exists C > 0\,:\,\frac{h_z^2}{\Delta t} \sim C, \qquad \text{for }h_z,\Delta t \rightarrow 0, \quad(N,N_t\rightarrow +\infty),
	\end{equation*}
	i.e., it is bounded whenever we impose the compatibility conditions for the relation between 
	the time and space discretization. We observe that $\prescript{{\text{\tiny ARIT}}}{}{D_N}$ is the
	$(N-2)$th diagonal sampling matrix for the function $K(p(z))$. Therefore, it is one of the sequences of
	which we know the distribution (\textbf{GLT3}), and we conclude that
	\[
	\{h_z^2\, \prescript{{\text{\tiny ARIT}}}{}{D_N} \}_N \sim_{\rm GLT} C s'(p(\psi(z))). \qedhere
	\]
\end{proof}

\begin{lemma}\label{lem:symboldiffusion}
	The matrix sequence $\{h_z^2\, \prescript{{\text{\tiny ARIT}}}{}{L_N}\}_N$ has GLT symbol
	\[ K(p(\psi(x))) (2-2\cos(\theta)),\]
	where $\psi(x)$ is the function mapping $[0,1]$ interval to the domain of definition for the $z$ variable.
\end{lemma}

\begin{proof}
	Let us start from the first tridiagonal matrix
	\[ A_N =  \frac{1}{2}\operatorname{tridiag}\left(-K\left[p_{i-1}^{(l+1)}\right],2 K\left[p_i^{(l+1)}\right],-K\left[p_{i+1}^{(l+1)}\right]\right)_{i=1}^{N-2}, \]
	to produce its GLT symbol, we can consider the matrix sequence
	\[ D_N(K(p^{(l+1)})) T_N(2-2\cos\theta).\] If we then perform a direct
	comparison of \[A_N -\frac{1}{2}D_N(K(p^{(l+1)}))
	T_N(2-2\cos\theta),\] we observe that the only nonzero entries are the
	ones on the lower and upper diagonals, in which we have the
	differences $K(p^{(l+1)}_{i\pm 1}) - K(p^{(l+1)}_{i})$. From this, we
	bound the modulus of each off-diagonal entries of \[A_N
	-\frac{1}{2}D_N(K(p^{(l+1)})) T_N(2-2\cos\theta),\] by the modulus of
	continuity of $\frac{1}{2}\omega_{K(p(z))}(h_z)$ and then the $1$-norm
	and the $\infty$-norm of the difference $A_N
	-\frac{1}{2}D_N(K(p^{(l+1)})) T_N(2-2\cos\theta)$ are bounded by
	$\omega_{K(p(z))}(h_z)$. Thus,
	
	\begin{equation*}
	\| A_N -\frac{1}{2}D_N(K(p^{(l+1)})) T_N(2-2\cos\theta) \| \leq
	\omega_{K(p(z))}(h_z) \rightarrow 0 \text{ for } N \rightarrow
	+\infty,
	\end{equation*}
	therefore, $Z_N = A_N -\frac{1}{2}D_N(K(p^{(l+1)}))
	T_N(2-2\cos\theta)$ is distributed as zero. Then, a direct application
	of \textbf{GLT3} and \textbf{GLT4} tells us that $A_N \sim_{\rm GLT}
	1/2 K(p(\psi(x)))(2-2\cos(\theta))$. With minor modifications, the
	same arguments can be applied to the other parts of $\{h_z^2\,
	\prescript{{\text{\tiny ARIT}}}{}{L_N} \}_N$, and by means of the
	*-algebra properties from \textbf{GLT4}, we can write
	\begin{equation*}
	\begin{split}
	\{ h^2_z \prescript{{\text{\tiny ARIT}}}{}{L_N}\}_N \sim_{\rm GLT} & \frac{1}{2}K(p(\psi(x)))(2-2\cos(\theta)) \\ & + \frac{1}{2} K(p(\psi(x)))(2-2\cos(\theta))(1-e^{-i\theta}) \\ & + \frac{1}{2}K(p(\psi(x)))(2-2\cos(\theta))(1-e^{i\theta})\\ & +
	\frac{1}{2}p(\psi(x))K'(p(\psi(x)))(2-2\cos(\theta)) \\
	& + \frac{1}{2}p(\psi(x))K'(p(\psi(x)))(-1 + e^{i\theta}) \\ & +\frac{1}{2}p(\psi(x))K'(p(\psi(x)))(-1 + e^{i\theta}) \\
	= & K(p(\psi(x))) (2-2\cos(\theta). \qedhere
	\end{split}
	\end{equation*}
\end{proof}

\begin{remark}\label{rmk:onlythelaplacianmatters}
	In the computation of the GLT symbol of the sequence $\prescript{{\text{\tiny ARIT}}}{}{L_N}$ we have 
	that a part of the distribution simplifies itself. This appears as if there was a simplification 
	of the lower order terms induced by the arithmetic mean, i.e., in a \textit{chain-rule} 
	style the correction due to the term $p K'(p)$ cancels out.
\end{remark}

\begin{lemma}\label{lem:symboltransport}
	The matrix sequence $\{h_z^2\, \prescript{{\text{\tiny ARIT}}}{}{T_N}\}_N$ is  a sequence distributed as zero.
\end{lemma}

\begin{proof}
	By construction, the matrices of the sequence $\{h_z^2\, \prescript{{\text{\tiny ARIT}}}{}{T_N}\}_N$ 
	are such that
	\begin{equation*}
	\| h_z^2\, \prescript{{\text{\tiny ARIT}}}{}{T_N} \| \leq h_z \|K'(p)\|_{\infty} \leq \frac{C}{N},
	\end{equation*}
	for some constant $C$ independent of $N$. Therefore, $\{h_z^2\, \prescript{{\text{\tiny ARIT}}}{}{T_N}\}_N \sim_{\sigma} 0$.
\end{proof}

\begin{theorem}\label{thm:eigenvaluedistributionwitharithmeticmean}
	The matrix sequence $\{h_z^2\, \prescript{{\text{\tiny ARIT}}}{}{J_N}\}_N$ is distributed in
	the eigenvalue sense as the function
	\begin{equation*}
	\kappa(x,\theta) = C s'(p(\psi(x))) + K(p(\psi(x))) (2-2\cos(\theta)),
	\end{equation*}
	where $\psi(x)$ is the function mapping $[0,1]$ interval to the domain of definition for the $x$ variable.
\end{theorem}

\begin{proof}
	The proof follows in two step from applying \textbf{GLT4} with the symbols obtained in Lemmas~\ref{lem:symboltime}, \ref{lem:symboldiffusion}, and~\ref{lem:symboltransport}, we first find that $\{h_z^2\, \prescript{{\text{\tiny ARIT}}}{}{J_N}\}_N\sim_{\rm GLT} \kappa(z,\theta) = C s'(p(\psi(x))) + K(p(\psi(x))) (2-2\cos(\theta))$. Then, by \textbf{GLT2}, that holds in virtue of Lemma~\ref{lem:symboltransport}, we have that the distribution holds also in the eigenvalue sense.
\end{proof}

\begin{remark}
	From Theorem~\ref{thm:eigenvaluedistributionwitharithmeticmean} we
	infer that the ill-conditioning in the Jacobian matrices is
	determined by two main factors. On the one hand, we have the
	ill-conditioning due to the diffusion operators that drives the
	eigenvalues to zero as an $O(h_z^2)$. On the other, we observe that this
	effect is enhanced by the behavior of the hydraulic conductivity $K(p)$ at
	the current time step/Newton iterate. Indeed, for lower values of the
	pressure head, this further enhances the decay to zero of the eigenvalues.
	This implies also that, when the soil becomes more saturated, the
	ill-conditioning is due only to the diffusion.
\end{remark}

Now, with the choice of this average for all the spatial dimensions, we
can explicitly write the nonzero elements of the Jacobian matrix by
maintaining the local $(i,j,k)$-indices, i.e., in a way that is
indepedent of the selected ordering; see the supplementary materials
Section~\ref{sec:jacobian_arithmetic} for the complete expression. To
complete the discretization we only need  to impose and discretize the
boundary conditions. For the test case considered here we focus on
Dirichlet boundary conditions that can be either homogeneous, to identify
water table boundary conditions, or time dependent.
\begin{theorem}\label{thm:spectraldistribution_jacobians_3D_arithmetic_mean}
	The sequence $\{ J_{\mathbf{N}}^{(k,j)} \}_\mathbf{N}$ obtained with
	the entries in~\eqref{eq:jacobian_arithmeticaverage}, for $K(p)$,
	$s(p)$ in~\eqref{eq:s(p)}, \eqref{eq:K(p)} is distributed in the
	sense of the eigenvalues as the function
	\begin{equation*}
	\begin{split}
	f(\mathbf{x},\theta) & = C \rho \phi s'(\mathbf{p}^{(k,j)}(\psi(\mathbf{x}))) \qquad \mathbf{x} \in [0,1]^3, \; \mathbf{\theta} \in [-\pi,\pi]^3, \\ & + K(\mathbf{p}^{(k,j)}(\psi(\mathbf{x})))(8 - 2\cos(\theta_1) - 2\cos(\theta_2) - 2\cos(\theta_3)),
	\end{split}
	\end{equation*}
	where $\psi(\mathbf{x})$ is the function mapping the $[0,1]^3$ cube to the physical domain, and $C = \lim_{\mathbf{N},N_T \rightarrow \infty} \frac{\mathbf{h}}{\Delta t}$.
\end{theorem}

\begin{proof}[Proof of Theorem~\ref{thm:spectraldistribution_jacobians_3D_arithmetic_mean}]
	Moving from the results in
	Theorem~\ref{thm:eigenvaluedistributionwitharithmeticmean} to the
	ones in
	Theorem~\ref{thm:spectraldistribution_jacobians_3D_arithmetic_mean}
	is just a technical rewriting. The proofs in
	Lemma~\ref{lem:symboltime}, and Lemma~\ref{lem:symboltransport}
	remain the same. We need to rewrite the decomposition of
	$\prescript{{\text{\tiny ARIT}}}{}{L_N}$ exploiting the fact that
	the tensor structure in the grid corresponds to a Kronecker structure
	in the matrix.
\end{proof}

We can prove the exact same statement given in
Theorem~\ref{thm:spectraldistribution_jacobians_3D_arithmetic_mean}
also for the upstream mean together with the expression of the complete
Jacobian. By virtue of the mechanism applied, the proof is completely analogous and requires 
only some technical adjustments related to the corrections in rank. The details and full Jacobian 
expression for this case are given in the supplementary materials, see  
sections~\ref{sec:upstream_average} and~\ref{sec:jacobian-expression}.

\begin{theorem}\label{thm:spectraldescription}
	The sequence $\{ J_{\mathbf{N}}^{(k,j)} \}_\mathbf{N}$ obtained with the choice of the upstream average (entries in~\eqref{eq:jacobian_upstream_mean}) or with the arithmetic average (entries in~\eqref{eq:jacobian_arithmeticaverage}), for $K(p)$, $s(p)$ in~\eqref{eq:s(p)}, \eqref{eq:K(p)} is distributed in the sense of the eigenvalues as the function
	\begin{equation*}
	\begin{split}
	f(\mathbf{x},\theta) & = C \rho \phi s'(\mathbf{p}^{(k,j)}(\psi(\mathbf{x}))) \qquad \mathbf{x} \in [0,1]^3, \; \mathbf{\theta} \in [-\pi,\pi]^3, \\ & + K(\mathbf{p}^{(k,j)}(\psi(\mathbf{x})))(8 - 2\cos(\theta_1) - 2\cos(\theta_2) - 2\cos(\theta_3)),
	\end{split}
	\end{equation*}
	where $\psi(\mathbf{x})$ is the function mapping the cube $[0,1]^3$
	to the physical domain and $C = \lim_{\mathbf{N},N_T \rightarrow
		\infty} \frac{\mathbf{h}}{\Delta t}$.
\end{theorem}

We can conclude that, independently of the mean used, the
\emph{asymptotic} behavior of the spectrum remains the same and it is dominated 
by the diffusion operator. This
suggests the idea of how to define an appropriate sequence of
preconditioners for the underlying problem. 
Specifically, we use the matrix sequence generated by 
	the considered discretization of the Jacobians of~\eqref{eq:dynamic_of_the_problem}
	having modified the flux term $q$ to be 
	\[q_{\text{prec.}} = - K(p) \nabla p .\]
	In practice, we disregard all the terms related to the fluid motion along the $\hat{z}$ axis thus using the Jacobians of the
	nonlinear diffusion operators. using the Jacobian matrices of the nonlinear diffusion operators only.
Indeed, these diffusion matrices have the same spectral distribution we have proved with Theorem~\ref{thm:spectraldescription}, and therefore by \textbf{GLT5} and \textbf{GLT4} they provide an optimal preconditioner for the underlying sequence of matrices.
We note that, to the  best of our knowledge, this is the first theoretical justification of a very 
common approach in preconditioning Newton correction linear systems for solving the Richards
equation\cite{JonesWoodward01,kollet2005,maxwell2013}.
We are now left with the problem of inverting the
asymptotically spectrally equivalent sequence we have just built from
the nonlinear diffusion. To face this task, we will perform a further
approximation of the theoretical operators sequence by employing the preconditioning library 
available in PSCToolkit, a software framework recently proposed for scalable simulations on high-end
supercomputers.\footnote{The libraries can be obtained from~\url{https://psctoolkit.github.io/}}

\section{Software framework for very large-scale simulations}\label{sec:softwareframework}

In this work we employ the recently proposed PSCToolkit software framework for parallel 
sparse computations on current petascale supercomputers. 
PSCToolkit is composed of two main libraries, named PSBLAS (Parallel
Sparse Basic Linear Algebra Subprograms)~\cite{FC2000,FB2012}, and
AMG4PSBLAS (Algebraic MultiGrid Preconditioners for
PSBLAS)~\cite{dambra2020b}. PSBLAS implements all the main computational
building blocks for iterative Krylov subspace linear solvers on parallel
computers made of multiple nodes; a plugin for NVIDIA GPUs allows the
exploitation of these devices in main sparse matrix and vector
computations on hybrid architectures. The toolkit also implements a number
of support functionalities to handle the construction of sparse matrices
and of their communication structure.

AMG4PSBLAS is a package of preconditioners leveraging on PSBLAS kernels
and providing one-level Additive Schwarz (AS) and Algebraic MultiGrid
(AMG) preconditioners for PSBLAS-based Krylov solvers.
For the  present work, we developed a new set of interfaces allowing a
seamless integration of linear solvers and preconditioners from PSBLAS and
AMG4PSBLAS into the SUNDIALS/KINSOL library~\cite{hindmarsh2005sundials},
in order to  use the KINSOL version of the modified Newton methods to
solve the discretized Richards nonlinear equations. 
Interfacing with KINSOL provides a double advantage: the extension of PSCToolkit to handle
non-linear algebraic equations, as well as the extension of the KINSOL library with new methods 
for solving sparse linear systems on high-end supercomputers.

In the following we describe the details of the proposed parallel Richards solver 
implemented in C by using PSCToolkit facilities for data generation and distribution 
and for preconditioner setup and application inside Krylov solvers, and on top of KINSOL facilities 
for the quasi-Newton procedure as described in Section~\ref{sec:inexact-newton}.

\subsection{Domain decomposition and data partitioning}\label{sec:data-distribution} 

In our parallel solution procedure of the discrete problem
in~\eqref{eq:dynamic_of_the_problem}, we employ a 2D block decomposition
of the parallelepipedal domain $\Omega$ of  size
$[0,L_x]\times[0,L_y]\times[0,L]$, i.e., we partition only in the
horizontal directions. Indeed, in realistic simulations, one often needs to
work with a wide horizontal domain, with a fixed size and resolution for
the vertical layer and the vertical physical fields. The above domain
decomposition corresponds to a mesh partitioning in which each parallel
process owns all of degrees of freedom (dofs) in the $z$ direction, while uniform partitioning
is applied in the $x$ and $y$ directions. PSBLAS provides all the
functionalities needed for very general parallel mesh partitioning and to set
up parallel data structures by compact and easy to use interfaces. All
PSBLAS routines have a pure algebraic interface, where the main data structures
are a \textit{distributed sparse matrix} and a 
corresponding \textit{communication descriptor}. The procedure for mesh 
partitioning builds a local sparse matrix of possibily non-contiguous rows of the global matrix, 
where each row corresponds to a local mesh dof and is handled through a local numbering scheme, and
automatically builds an index map object contained in the communication descriptor to keep track 
of the correspondence between local and global indices. The communication
descriptor also includes information needed for data communication
among processes which is  automatically generated  after the mesh
partitioning procedure and handled internally by the software. 

\subsection{AMG4PSBLAS preconditioners}
\label{sec:preconditioners}

The results in 
Theorem~\ref{thm:spectraldistribution_jacobians_3D_arithmetic_mean}
suggest that we can expect to achieve good convergence in the
iterative solution of the linear
systems of the quasi-Newton steps~\eqref{eq:linearsystemtosolve} by using a matrix $M$ with symbol
given by
$f(\mathbf{x},\boldsymbol{\theta})$ as a preconditioner. By means of
the $*$-algebra properties in \textbf{GLT4} and \textbf{GLT5}, this
would guarantee a sequence $\{M^{-1}
J_{\boldsymbol{\Phi}}(\mathbf{p}^{(l)}) \}_{\mathbf{N}} \sim_{\lambda} 1$.
Informally, by the spectral symbol, this implies that we can expect the
spectrum of the sequence of the preconditioned matrices to have a cluster
of eigenvalues in $1$. While this would guarantee the theoretical
properties we look for, we also need to approximate the inverse of the
symmetric and positive definite (\emph{spd} for short) preconditioner. To this aim, 
we exploit some of the methods available in AMG4PSBLAS~\cite{dambra2020b}.

AMG4PSBLAS includes iterative methods for the approximation of the inverse
of the preconditioner $M$, including domain decomposition techniques of
Additive Schwarz (AS) type and AMG methods based on aggregation of
unknowns; details on the methods and their parallel versions are available in
the literature~\cite{BDDF2007,DDF2010,dambra2020b}. In the following we describe the
main features of the preconditioners  selected for the experiments
discussed in section~\ref{sec:numerical_experiment}. 

In the AS methods, the index space, i.e., the set of row/column indices of
$M$, $\Omega^N = \{1, 2, \ldots, N \}$, is divided into $m$, possibly
overlapping, subsets $\Omega^N_i$ of size $N_i$. For each subset, we can
define the restriction operator $R_i$ which maps a vector $x \in
\mathcal{R}^N$ to the vector $x_i \in \mathcal{R}^{N_i}$ made of the
components of $x$ having indices in $\Omega^N_i$, and the corresponding
prolongation operator $P_i=(R_i)^T$. The restriction of $M$ to the
subspace $\Omega^N_i$ is then defined by the Galerkin product
$M_i=R_iMP_i$. The Additive Schwarz preconditioner for $M$ is defined as
the following matrix:
\begin{equation}
\label{AS}
M^{-1}_{AS}= \sum_{i=1}^{m} P_i (M_i)^{-1} R_i,
\end{equation}
where $M_i$ is supposed to be nonsingular and an inverse (or an
approximation of it) can be computed by an efficient algorithm. We observe that in the case of non-overlapping subsets 
$\Omega^N_i$, the AS preconditioner reduces to the well-known block-Jacobi preconditioner.

AMG4PSBLAS includes all the functionalities for setup and application of
the preconditioner in~\eqref{AS} in a parallel setting, where the original
index set $\Omega$ has been partitioned into subsets $\Omega_i$, each of
which is owned by one of the $m$ parallel processes, so that the inverse
of $M_i$ can be locally computed by a LU factorization or variants of incomplete LU 
factorizations and sparse approximate inverses. 

AS methods can also be applied as smoothers in an AMG procedure, which is
the main focus of the AMG4PSBLAS software framework. In this case, the
inverse of the preconditioner matrix is defined by the recursion below. 
Let $M_l$ be a sequence of spd coarse matrices obtained from $M$ by a
Galerkin product $M_{l+1}=(P_l)^TM_lP_l, \ l=0, \ldots, n_l-1$, with
$M_0=M$, and $P_l$ a sequence of prolongation matrices of size $N_l \times
N_{l+1}$, with $N_0=N$ and $N_{l+1}<N_l$. Let $S_l$ be a convergent
smoother with respect to the $M_l-$inner product, the well known symmetric
V-cycle, with one smoother iteration applied at each level before and
after the coarse-level correction, defines the inverse of the
preconditioner matrix as a sequence of the following matrices:
\begin{equation*}
\label{Vcycle}
\overline{M}_l =  (S_l)^{-T}+(S_l)^{-1}-(S_l)^{-T}M_l(S_l)^{-1} +
(I-(S_l)^{-T}M_l)(P_l\overline{M}_{l+1}(P_l)^T)(I-M_l(S_l)^{-1}) \ \ \  \forall l,
\end{equation*}
assuming that $\overline{M}_{n_l} \approx (M_{n_l})^{-1}$ is an
approximation of the inverse of the coarsest-level matrix.
AMG methods are characterized by the coarsening procedure applied to setup
the sequence of coarse matrices, i.e., the corresponding prolongation
matrices; they use only information from the original (fine) matrix, with no additional information 
related to the geometry of the problem.
In AMG4PSBLAS two different parallel coarsening procedures for spd matrices are available;
both employ  disjoint aggregates of fine dofs to form the coarse dofs, and
the prolongation matrices are piecewise-constant interpolation matrices
(unsmoothed aggregation) or a smoothed variant thereof (smoothed
aggregation).
For details on the coarsening procedures implemented in AMG4PSBLAS are available in
the literature~\cite{DDF2010,dambra2020b}. Here we note that in section~\ref{sec:numerical_experiment} 
we refer to the two acronims:
\begin{description}
	\item[DSVMB:] the smoothed aggregation scheme introduced by Van\v{e}k, Mandel and Brezina in~\cite{VMB1996}, and applied in a parallel setting by a decoupled approach, where each process applies the coarsening algorithm to its subset of dofs, ignoring interactions with dofs owned by other processes~\cite{DDF2010}.
	\item[SMATCH:] the smoothed aggregation scheme introduced by D'Ambra and Vassilevski~\cite{DV2013,dambra2018,dambra2020b}. It relies on a parallel coupled aggregation of dofs based on a maximum weighted graph matching algorithm, where the maximum size of aggregates can be chosen in a flexible way by a user-defined parameter so that computational complexity and convergence properties of the final preconditioner can be balanced.  
\end{description}
The parallel smoothers available in AMG4PSBLAS can be applied both as one-level preconditioners as well as in
an AMG procedure. These include variants of the AS methods described above, weighted versions of the simple
Jacobi method, such as the so-called $\ell_1-$Jacobi, and a hybrid version of  Gauss-Seidel, which acts as 
the Jacobi method among matrix blocks assigned to different parallel processes and updates the unknowns in 
a \lq\lq Gauss-Seidel style\rq\rq~within the block owned by a single process~\cite{DF2016,dambra2020b}.
Some of the above methods are well suited for use with the PSBLAS plugin for GPU accelerators in 
the preconditioner application phase.

\section{Numerical experiments}
\label{sec:numerical_experiment}

In this section we investigate the parallel performances of 
our Richards equation solver relying on the PSCToolkit software framework using preconditioners discussed 
in Section~\ref{sec:preconditioners}. The section is divided in two parts: in Section~\ref{strongscaling} 
we analyse \emph{strong scalability} by measuring how the overall computational time decreases with 
the number of processes for a fixed problem size, then in Section~\ref{sec:weak_scalability} we focus on
\emph{weak scalability}, i.e., we measure how the solution time varies by increasing the number of processes,
while keeping fixed the problem size per process, so that the global problem size proportionally increases
with the number of processes. Experiments validating the spectral analysis are reported in the
\emph{supplementary materials} sections~\ref{theoreticalresults}, and~\ref{sec:algscalability}.

All the experiments are executed on the CPU cores of the \emph{Marconi-100 supercomputer} (21\textsuperscript{st} in the June 2022 TOP500 list~\footnote{\label{foot:top500}See the relevant list 
on the website \href{https://www.top500.org/lists/top500/list/2022/06/}{www.top500.org/}}), 
with no usage of hyperthreading. \emph{Marconi}'s nodes are built on an \emph{IBM Power System AC922}, 
they contain two banks of 16 cores \emph{IBM POWER93} 3.1 GHz processors and are equipped with 256 GB of RAM.
The inter-node communication is handled by a Dual-rail \emph{Mellanox EDR Infiniband} network by \emph{IBM}
with 220/300 GB/s of nominal and peak frequency. All the code is compiled with the \texttt{gnu/8.4.0} suite
and linked against the \texttt{openmpi/4.0.3} and \texttt{openblas/0.3.9} libraries. We use only inner
functionalities of the PSBLAS 3.7.0.2 and AMG4PSBLAS 1.0 libraries, with no use of optional third party
libraries.  

We discuss results of the overall solution procedure for solving the discretized Richards equation, by
comparing parallel performance and convergence behavior when the AS preconditioner in~\eqref{AS} and the 
two AMG preconditioners based on the two different coarsening strategies discussed
in Section~\ref{sec:preconditioners} are employed for solving the linear systems within the modified 
Newton procedure. 
For the one-level $AS$ preconditioner we use one layer of mesh points in each
direction as overlap among the subdomains $\Omega_i$, each of them
assigned to different processes, and apply an Incomplete LU
factorization with no fill-in for computing the local subdomain matrix
inverses $M_i^{-1}$.
In the case of AMG preconditioners, we apply a symmetric V-cycle with $1$
iteration of hybrid backward/forward Gauss-Seidel as pre/post-smoother at
the intermediate levels. As a coarsest-level solver we use a parallel
iterative procedure based on the preconditioned Conjugate Gradient method
with block-Jacobi as preconditioner, where ILU with 1 level of fill-in is
applied on the local diagonal blocks. The coarsest-level iterative
procedure is stopped when the relative residual is less than $10^{-4}$ or
when a maximum number of $30$ iterations is reached.
In both the DSVMB and SMATCH procedures, the coarsening is stopped when
the size of the coarsest matrix includes no more than $200$ dofs per core;
in the case of the SMATCH coarsening, a maximum size of aggregates equal
to $8$ is required.
In the following, the two AMG preconditioners are referred as VDSVMB and
VSMATCH, respectively.
In order to reduce the setup costs of the AMG preconditioners in
non-linear and/or time-dependent simulations, we support a re-use strategy
of operators that  is often used in
our Richards solution approach. Namely, we compute the multilevel
hierarchy of coarser matrices only on the first Jacobian of the sequence
(at the first time step), then for subsequent Jacobians we just update the
smoother on the current approximation of the matrix, while keeping fixed
the coarser matrices and transfer operators. Moreover, at each subsequent
time step, we reuse the Jacobian (and its related approximation) from the
previous one; we leave KINSOL control over the possible need  to recompute
a Jacobian at each new Newton iteration. 

All the preconditioners are applied as right preconditioner to the PSBLAS-based Restarted GMRES 
with restarting step equal to $10$ -- GMRES(10). We stop the iterations when the relative residual 
satisfies $ \|J \mathbf{d}_r + \boldsymbol{\Phi} \| < \eta \| \boldsymbol{\Phi}\|$
with $\eta = 10^{-7}$ or when a maximum number of iterations equal to 200 is done. 

To solve the Richards equation with upstream averages on a parallelepipedal domain $\Omega$ 
of size $[0,L_x]\times[0,L_y]\times[0,L]$, we apply water at height $z=L$
such that the pressure head becomes zero in a square region at the
center of the top layer, ($\frac{a}{4} \leq x \leq \frac{3a}{4}$,
$\frac{b}{4} \leq y \leq \frac{3b}{4}$), and is fixed to the value
$h=h_r$ on all the remaining boundaries, that is
\begin{equation*}
p(x,y,L,t) = \frac{1}{\alpha} \ln \left[ \exp(\alpha h_r) + (1-\exp(\alpha h_r)) \chi_{[\frac{a}{4},\frac{3a}{4}]\times[\frac{b}{4},\frac{3b}{4}]}(x,y,z) \right],
\end{equation*}
where we denote by $\chi_\Omega$ the characteristic function of the
set $\Omega$. The associated initial condition is given by
$p(x,y,z,0) = h_r$. In all cases we run the simulation for $t \in
[0,2]$ and $N_t = 10$. This means that the number of time steps
$N_t$ is fixed independently from the number of processes $n_p$, thus
the performance analysis will be done on the quantities averaged on
the number of time steps relative to the given $n_p$.

All the timings reported in the figures of the following sections are in seconds.

\subsection{Strong scalability}
\label{strongscaling}

In this section we discuss parallel performance results of our solution
procedure when we fix the target
domain as the parallelepiped $[0,64]\times[0,64]\times[0,1]$, discretized
with $N_x=N_y=800$ mesh points in the $x$ and $y$ directions, and $N_z=40$
mesh points in the vertical direction, for a total
number of $20$ millions of dofs, on a number of computational cores from
$1$ to $256$; in particular we used a number of cores $n_p=4^p \
p=0,\ldots,4$, so that the computational domain is uniformly partitioned
in an increasing number of vertical subdomains with square basis, for
increasing number of parallel cores.

We start by looking at the average number of linear iterations done for
the Newton correction by using the three different preconditioners, and
the time needed per each linear iteration in
Figure~\ref{fig:strong_scaling_itandtimeperiter}. We observe that, as
expected, the AMG preconditioners require a smaller number of iterations
than the AS method; the latter shows an increase in the number of iterations as 
the number of cores, and therefore of subdomains, increases. VDSVMB always
requires the smallest number of iterations showing the ability of the 
DSVMB coarsening procedure to setup a good quality matrix hierarchy on the
first Jacobian.
On the other hand, we observe that when  increasing the number of cores,
the number of linear iterations also increases, due to the decoupled
parallel approach of the coarsening.
The VSMATCH algorithm, thanks to its coupled approach, produces a
number of iterations that is essentially unaffected by the number of 
processes, even though  it is always larger than VDSVMB. 

If we look at the time per linear iteration in
Figure~\ref{fig:strong_scaling_itandtimeperiter} (b), we observe, as
expected, that the AS preconditioner has the smallest time per iteration.
Indeed, for his one-level nature, its application cost is smaller than that
of the multigrid methods. It also shows a regular decreasing for
increasing number of parallel cores. If we look at the time per iteration
of the AMG preconditioners, we observe that VDSVMB and VSMATCH have very
similar behavior, and a regular decreasing for larger number of cores. On
the other hand, VDSVMB always achieves a smaller time per iteration than
VSMATCH, due to its  better coarsening ratio. Indeed, it coarsens the 
original fine matrix in an efficient way, producing coarse matrices
that are both smaller and of good quality.
This feature makes the VDSVMB preconditioner  competitive with respect
to the AS method. Both preconditioners produce similar global solution
times, as shown in Figure~\ref{fig:strong_scaling_timeandspeedup}. We can
observe that, in all the cases, the overall simulation has a
regular decreasing computational time for increasing number of cores,
leading to a global speedup ranging from $150$ to $169$, depending on the
preconditioner, on $256$ computational cores. This corresponds to
a satisfactory parallel efficiency ranging from $59 \%$ to $66 \%$.
\begin{figure}[htbp]
	\centering
	\subfloat[Average number of linear iteration for the overall simulation]{
%
%
\definecolor{mycolor1}{rgb}{0.00000,0.44700,0.74100}%
\definecolor{mycolor2}{rgb}{0.85000,0.32500,0.09800}%
\definecolor{mycolor3}{rgb}{0.92900,0.69400,0.12500}%
\begin{tikzpicture}

\begin{axis}[%
width=0.25\columnwidth,
height=0.25\columnwidth,
at={(0.758in,0.481in)},
scale only axis,
xmode=log,
log basis x=2,
xmin=1,
xmax=256,
xtick={1,4,16,64,256},
xminorticks=true,
xlabel style={font=\color{white!15!black}},
xlabel={$n_p$},
ymin=60,
ymax=160,
axis background/.style={fill=white},
legend style={at={(1.00,1.00)}, legend cell align=left, align=left, draw=none, fill=none}
]
\addplot [color=mycolor1, line width=2.0pt, mark=o, mark options={solid, mycolor1}]
  table[row sep=crcr]{%
1	94.6\\
4	94.8\\
16	94.6\\
64	94.7\\
256	94.9\\
};
\addlegendentry{VSMATCH}

\addplot [color=mycolor2, line width=2.0pt, mark=x, mark options={solid, mycolor2}]
  table[row sep=crcr]{%
1	63.6\\
4	76.8\\
16	79.7\\
64	80\\
256	80.5\\
};
\addlegendentry{VDSVMB}

\addplot [color=mycolor3, line width=2.0pt, mark=triangle, mark options={solid, mycolor3}]
  table[row sep=crcr]{%
1	122.5\\
4	153.3\\
16	154.9\\
64	156.5\\
256	157.7\\
};
\addlegendentry{AS}

\end{axis}
\end{tikzpicture}%
}\hfil
	\subfloat[Execution Time per linear iteration $T(s)$]{
%
%
\definecolor{mycolor1}{rgb}{0.00000,0.44700,0.74100}%
\definecolor{mycolor2}{rgb}{0.85000,0.32500,0.09800}%
\definecolor{mycolor3}{rgb}{0.92900,0.69400,0.12500}%
\begin{tikzpicture}

\begin{axis}[%
width=0.25\columnwidth,
height=0.25\columnwidth,
at={(0.758in,0.481in)},
scale only axis,
xmode=log,
log basis x=2,
ymode=log,
xmin=1,
xmax=256,
xtick={1,4,16,64,256},
xminorticks=true,
xlabel style={font=\color{white!15!black}},
xlabel={$n_p$},
ymin=0.0106183636144578,
ymax=5.71802322410148,
axis background/.style={fill=white},
legend style={legend cell align=left, align=left, draw=none, fill=none}
]
\addplot [color=mycolor1, line width=2.0pt, mark=o, mark options={solid, mycolor1}]
  table[row sep=crcr]{%
1	5.71802322410148\\
4	1.51458374894515\\
16	0.495896086997886\\
64	0.124110330411827\\
256	0.0369711592202318\\
};
\addlegendentry{VSMATCH}

\addplot [color=mycolor2, line width=2.0pt, mark=asterisk, mark options={solid, mycolor2}]
  table[row sep=crcr]{%
1	4.36444878301887\\
4	1.15270752994792\\
16	0.358082429987453\\
64	0.087909846375\\
256	0.0258484281490683\\
};
\addlegendentry{VDSVMB}

\addplot [color=mycolor3, line width=2.0pt, mark=x, mark options={solid, mycolor3}]
  table[row sep=crcr]{%
1	1.75900547918367\\
4	0.460504128049576\\
16	0.156899446481601\\
64	0.0365616465367412\\
256	0.0106183636144578\\
};
\addlegendentry{AS}

\end{axis}
\end{tikzpicture}%
}
	\caption{Strong scaling. Average number of linear iterations and execution time per linear iteration with different preconditioners.}
	\label{fig:strong_scaling_itandtimeperiter}
\end{figure}
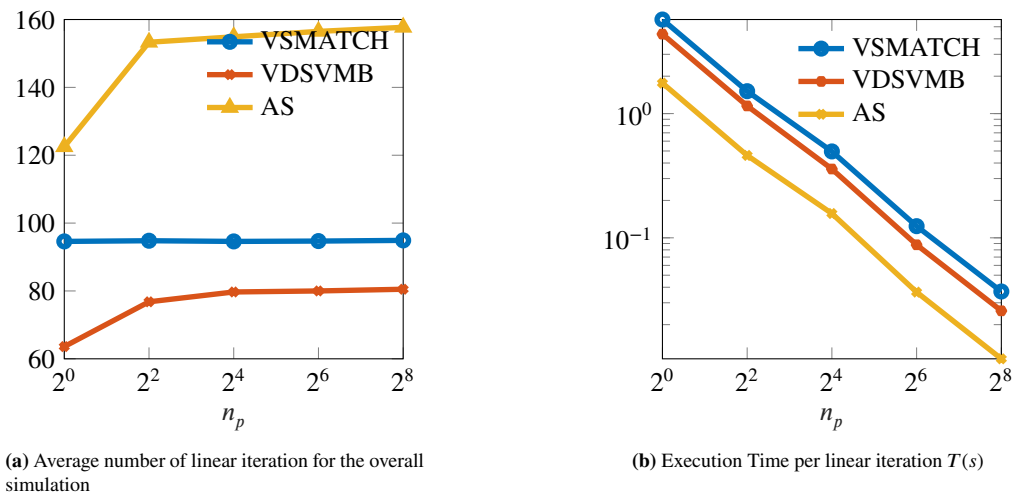
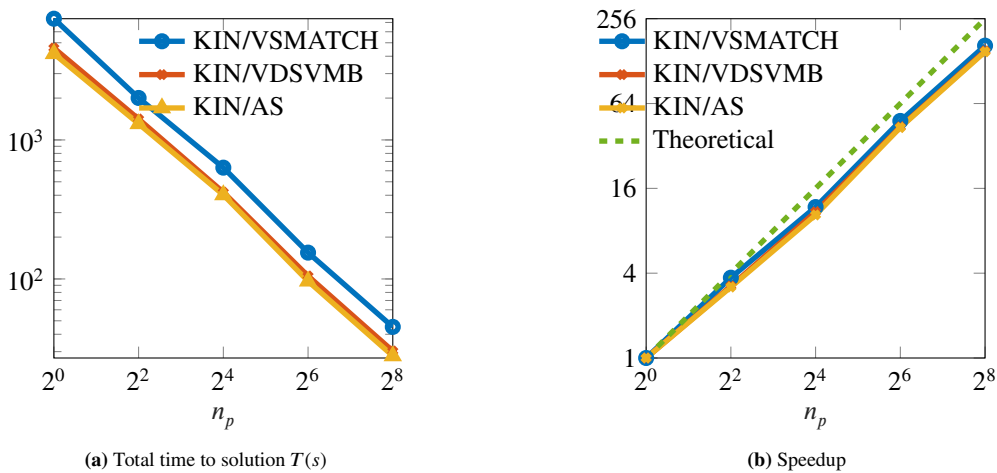
\begin{figure}[htbp]
	\centering
	\subfloat[Total time to solution $T(s)$]{
%
%
\definecolor{mycolor1}{rgb}{0.00000,0.44700,0.74100}%
\definecolor{mycolor2}{rgb}{0.85000,0.32500,0.09800}%
\definecolor{mycolor3}{rgb}{0.92900,0.69400,0.12500}%
\begin{tikzpicture}

\begin{axis}[%
width=0.25\columnwidth,
height=0.25\columnwidth,
at={(0.758in,0.481in)},
scale only axis,
xmode=log,
log basis x=2,
xmin=1,
xmax=256,
xminorticks=true,
xtick={1,4,16,64,256},
xlabel style={font=\color{white!15!black}},
xlabel={$n_p$},
ymode=log,
ymin=27,
ymax=7448.468432,
yminorticks=true,
axis background/.style={fill=white},
legend style={at={(1.00,1.00)}, legend cell align=left, align=left, draw=none, fill=none}
]
\addplot [color=mycolor1, line width=2.0pt, mark=o, mark options={solid, mycolor1}]
  table[row sep=crcr]{%
1	7448.468432\\
4	2004.386082\\
16	631.861923\\
64	154.674839\\
256	45.20514\\
};
\addlegendentry{\textsc{KIN}/VSMATCH}

\addplot [color=mycolor2, line width=2.0pt, mark=x, mark options={solid, mycolor2}]
  table[row sep=crcr]{%
1	4584.770917\\
4	1413.687056\\
16	420.631367\\
64	104.982019\\
256	30.307627\\
};
\addlegendentry{\textsc{KIN}/VDSVMB}

\addplot [color=mycolor3, line width=2.0pt, mark=triangle, mark options={solid, mycolor3}]
  table[row sep=crcr]{%
1	4168.530002\\
4	1306.492317\\
16	402.325916\\
64	96.355222\\
256	27.94471\\
};
\addlegendentry{\textsc{KIN}/AS}

\end{axis}
\end{tikzpicture}%
}\hfil
	\subfloat[Speedup]{
%
%
\definecolor{mycolor1}{rgb}{0.00000,0.44700,0.74100}%
\definecolor{mycolor2}{rgb}{0.85000,0.32500,0.09800}%
\definecolor{mycolor3}{rgb}{0.92900,0.69400,0.12500}%
\definecolor{mycolor4}{rgb}{0.45000,0.69400,0.12500}%
\begin{tikzpicture}

\begin{axis}[%
width=0.25\columnwidth,
height=0.25\columnwidth,
at={(0.758in,0.481in)},
scale only axis,
xmode=log,
log basis x=2,
log basis y=2,
xmin=1,
xmax=256,
xminorticks=true,
xtick={1,4,16,64,256},
xlabel style={font=\color{white!15!black}},
xlabel={$n_p$},
ymode=log,
ymin=1,
ymax=256,
yminorticks=true,
ytick={1,4,16,64,256},
yticklabels={1,4,16,64,256},
axis background/.style={fill=white},
legend style={at={(0.61,1.00)}, legend cell align=left, align=left, draw=none, fill=none}
]
\addplot [color=mycolor1, line width=2.0pt, mark=o, mark options={solid, mycolor1}]
  table[row sep=crcr]{%
1	1\\
4	3.71608468991554\\
16	11.7881267423674\\
64	48.1556566029463\\
256	164.770387438243\\
};
\addlegendentry{\textsc{KIN}/VSMATCH}

\addplot [color=mycolor2, line width=2.0pt, mark=x, mark options={solid, mycolor2}]
  table[row sep=crcr]{%
1	1\\
4	3.24313001066341\\
16	10.8997361506804\\
64	43.6719636436026\\
256	151.27449328184\\
};
\addlegendentry{\textsc{KIN}/VDSVMB}

\addplot [color=mycolor3, line width=2.0pt, mark=x, mark options={solid, mycolor3}]
  table[row sep=crcr]{%
1	1\\
4	3.19062726030558\\
16	10.3610775150761\\
64	43.2621078077118\\
256	149.170630219458\\
};
\addlegendentry{\textsc{KIN}/AS}

\addplot [color=mycolor4, line width=2.0pt, dashed]
  table[row sep=crcr]{%
1	1\\
4	4\\
16	16\\
64	64\\
256	256\\
};
\addlegendentry{Theoretical}

\end{axis}
\end{tikzpicture}%
}
	\caption{Strong scaling. Total execution time and Speedup for the solution of the whole problem (all time steps), \textsc{Kinsol} with different linear solvers.}
	\label{fig:strong_scaling_timeandspeedup}
\end{figure}

For the sake of completeness, in Table~\ref{tab:strong-scaling-iterations} we report the total number of computed Jacobians and of the Newton iterations required by the global simulation, when the different preconditioners are employed. We can see that the choice of the preconditioner also affects the KINSOL nonlinear procedure, and that the best behavior always corresponds to the VDSVBM preconditioner.
\begin{table}[htbp]
	\centering
	\begin{tabular}{||c||c|c||c|c||c|c||}
		\toprule
		& \multicolumn{2}{c||}{VDSVBM} & \multicolumn{2}{c||}{VSMATCH} & \multicolumn{2}{c||}{AS} \\
		\midrule
		$n_p$   & N Jac.s & NLin It.s & N Jac.s & NLin It.s & N Jac.s & NLin It.s  \\
		\midrule
		1	&3	&36	&3	&38	&3	&43\\
		4	&3	&37	&3	&38	&4	&39\\
		16	&3	&37	&3	&38	&4	&39\\
		64	&3	&37	&3	&38	&4	&39\\
		256	&3	&37	&3	&38	&4	&39\\
		\bottomrule
	\end{tabular}
	\caption{Strong scaling. Number of nonlinear iterations (NLin It.s), and number of computed Jacobians (N Jac.s) for the three preconditioners.}
	\label{tab:strong-scaling-iterations}
\end{table}

\subsection{Weak scalability}\label{sec:weak_scalability}

In this section we look at the \emph{weak
	scalability} of the overall solution procedure. 
The ideal goal  is to obtain  a constant  execution
time when the number of cores increases, while the computational work per
core is kept fixed. 
Unfortunately, it is in general  impossible to get an exactly 
constant execution  time. Indeed, we would need to have a perfect 
\emph{algorithmic} scalability of the preconditioners, i.e., 
the ability to keep the number of  linear iterations perfectly 
constant with an  increasing problem size. 
Moreover, the parallelization efficiency  of sparse matrix kernels, 
which are communication bound, is also affected by the increase in  the
number of processes and hence in the communication requirements. 

To achieve a scaling of the computational work that is meaningful with
respect to the physical properties of the underlying problem, we consider
a growing domain 
$\Omega(n_p) = [0,2^p\times 4.0]\times [0,2^q \times 4.0]\times[0,1.0]$
splitted on $n_p = p \times q$ processes for increasing $p=0,\ldots,7$,
$q=0, \ldots, 6$, and a corresponding mesh with a total of points equal to
$N(p \times q) = (2^p N_x,2^q N_y,N_z)$, where $N_x = N_y = 50$, and $N_z
= 40$. In this way, each process has the same amount of computational and
communication load, and the overall problem size increases with the total
number of processes $n_p$. We run experiments up to $8192$ CPU cores for
solving a problem with a global size up to about $829$ millions of dofs.

As in the previous section, we first discuss convergence behavior and
efficiency of the linear solvers for the Newton corrections, when the
different preconditioners are employed. In
Figure~\ref{fig:alg_scalwithupdate} we show the average number of linear
iterations along the overall simulation and the average execution time per
each linear iteration.
\begin{figure}[htbp]
	\centering
	\subfloat[Average number of linear iteration for the overall simulation]{
%
%
\definecolor{mycolor1}{rgb}{0.00000,0.44700,0.74100}%
\definecolor{mycolor2}{rgb}{0.85000,0.32500,0.09800}%
\definecolor{mycolor3}{rgb}{0.92900,0.69400,0.12500}%
\definecolor{mycolor4}{rgb}{0.49400,0.18400,0.55600}%
\begin{tikzpicture}

\begin{axis}[%
width=0.25\columnwidth,
height=0.25\columnwidth,
at={(0in,0in)},
scale only axis,
xmode=log,
log basis x=2,
xmin=1,
xmax=8192,
xminorticks=true,
xlabel style={font=\color{white!15!black}},
xlabel={$n_p$},
ymin=60,
ymax=200,
axis background/.style={fill=white},
legend style={at={(1.00,1.00)}, legend cell align=left, align=left, draw=none, fill=none}
]
\addplot [color=mycolor1, mark=o, line width=2.0pt, mark options={solid, mycolor1}]
  table[row sep=crcr]{%
1	94.2\\
4	105.6\\
16	101.7\\
64	97.3\\
256	94.9\\
1024	92.5\\
4096	90.1\\
8192	88.6\\
};
\addlegendentry{VSMATCH}

\addplot [color=mycolor2, mark=x, mark options={solid, mycolor2},line width=2.0pt]
  table[row sep=crcr]{%
1	67.4\\
4	84.8\\
16	82.6\\
64	81.7\\
256	80.5\\
1024	83.7\\
4096	86.7\\
8192	83.5\\
};
\addlegendentry{VDSVMB}

\addplot [color=mycolor3, mark=triangle, mark options={solid, mycolor3},line width=2.0pt]
  table[row sep=crcr]{%
1	104.1\\
4	165\\
16	186.5\\
64	163.6\\
256	157.7\\
1024	149.5\\
4096	196.4\\
8192	202.5\\
};
\addlegendentry{AS}


\end{axis}
\end{tikzpicture}%
}\hfil
	\subfloat[Execution Time per linear iteration $T(s)$]{
%
%
\definecolor{mycolor1}{rgb}{0.00000,0.44700,0.74100}%
\definecolor{mycolor2}{rgb}{0.85000,0.32500,0.09800}%
\definecolor{mycolor3}{rgb}{0.92900,0.69400,0.12500}%
\begin{tikzpicture}

\begin{axis}[%
width=0.25\columnwidth,
height=0.25\columnwidth,
at={(0in,0in)},
scale only axis,
xmode=log,
log basis x=2,
xlabel style={font=\color{white!15!black}},
xlabel={$n_p$},
xmin=1,
xmax=10000,
xminorticks=true,
ymin=0.005,
ymax=0.1,
ytick={0.1,0.01},
yticklabels={$10^{-1}$,$10^{-2}$},
ymode=log,
axis background/.style={fill=white},
legend style={at={(0.37,0.97)}, anchor=north west, legend cell align=left, align=left, draw=none, fill=none}
]
\addplot [color=mycolor1, line width=2.0pt, mark=o, mark options={solid, mycolor1}]
  table[row sep=crcr]{%
1	0.0201262717728238\\
4	0.022874977717803\\
16	0.0289177583284169\\
64	0.0314599091469681\\
256	0.0355911576817703\\
1024	0.0417667299459459\\
4096	0.052545189900111\\
8192	0.0495846750564334\\
};
\addlegendentry{VSMATCH}

\addplot [color=mycolor2, line width=2.0pt, mark=asterisk, mark options={solid, mycolor2}]
  table[row sep=crcr]{%
1	0.0168532215875371\\
4	0.0182821266627359\\
16	0.0213127797094431\\
64	0.0225124816156671\\
256	0.0269114677763975\\
1024	0.031553233572282\\
4096	0.034905747716263\\
8192	0.0350518509580838\\
};
\addlegendentry{VDSVMB}

\addplot [color=mycolor3, line width=2.0pt, mark=x, mark options={solid, mycolor3}]
  table[row sep=crcr]{%
1	0.00660769804034582\\
4	0.00713086487272727\\
16	0.00838831989812332\\
64	0.00950569971882641\\
256	0.0107557426062143\\
1024	0.0127743732173913\\
4096	0.0150891994348269\\
8192	0.0157516675604938\\
};
\addlegendentry{AS}

\end{axis}
\end{tikzpicture}%
}
	
	\caption{Weak scaling. Average number of linear iterations and execution time per linear iteration with different preconditioners.}
	\label{fig:alg_scalwithupdate}
\end{figure}
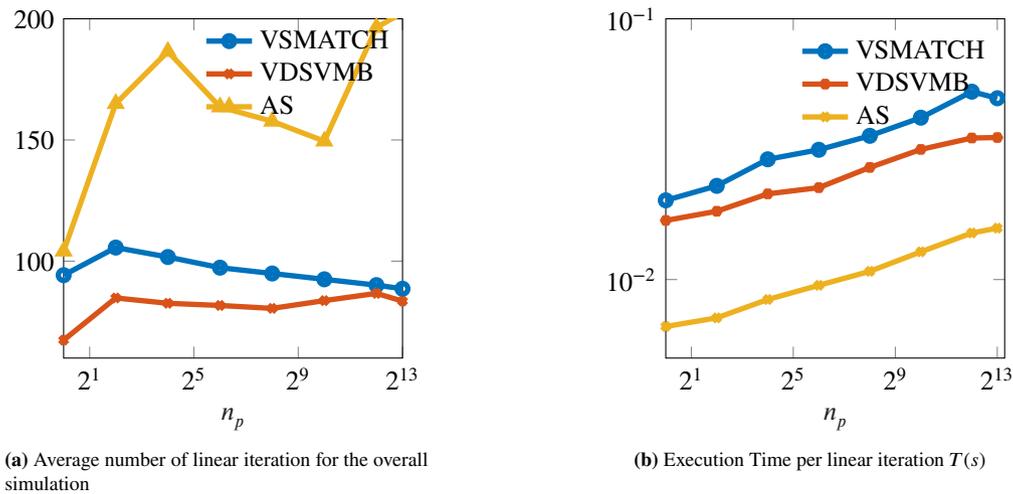
We observe that, as expected, the AS method has the worst algorithmic
scalability, showing a general increase of the linear
iterations needed to solve problems of increasing size. The average number
of iterations has a slow and small decrease from $16$ to $1024$ cores,
while there is a rapid increase from $150$ to $203$ iterations going from
$1024$ up to $8192$ cores, confirming that convergence properties of
the one-level Schwarz-type domain decomposition methods are dependent on
the number of involved subdomains. On the other hand, if we look at the
AMG preconditioners, where a better coupling among the subdomains is
considered, we see that the average number of iterations is 
smaller and has a very small variation for increasing number of 
subdomains. In particular, we observe that for the VDSVMB preconditioner
the average number of linear iterations has an increase from $67$ to $84$
iterations going from $1$ to $8192$ cores. The VSMATCH preconditioner, 
after an initial small increase, displays a decrease 
which shows that the coupled aggregation algorithm based on matching is
able to produce an effective preconditioner with very good algorithmic
scalability properties. This makes VSMATCH promising for
exploring extreme scalability in this type of simulation procedure for the
Richards model. 

If we look at the average execution time per linear iteration, for all the
preconditioners we observe a very small increase, which demonstrates a
good implementation scalability for all the sparse computations involved
in the linear solver phase. 

We now turn to the evaluation of the overall time to solution and \emph{scaled
efficiency}\footnote{Scaled efficiency is defines as $E_p=\nicefrac{T_1(N)}{T_p(N\cdot n_p)}$, where $T_1(N)$
is the time employed to solve a problem of size $N$ on 1 processing unit, and $T_p(N\cdot n_p)$ is the 
time employed for solving a problem of size $N \times n_p$ on $n_p$ processing units.}, 
of the global simulation.  
In  Figure~\ref{fig:implementationscalability}(a) we see that, as for linear
solvers, also for the global procedure 
there is a general small increase for increasing number of cores and
problem size for all the employed preconditioners, except for the VSMATCH
method going from $4096$ to $8192$ cores, where a better convergence
behavior of the linear solver also produces a small reduction in nonlinear
iterations (cfr Table~\ref{tab:weak-scaling-iterations}
in the supplementary materials). The best time to solution is in general
obtained by using the AS preconditioner till to $1024$ cores, due to the
smallest computational cost of this preconditioner. On the other hand, its
worst algorithmic scalability (i.e., the large increase in the number of
linear iterations) results in larger global execution times with respect
to the most effective VDSVMB preconditioner when the number of cores
increases. 

If we look at the scaled efficiency shown in percentage scale in Figure~\ref{fig:implementationscalability},
the AMG preconditioners show a similar behavior, while VSMATCH confirms his better efficiency when the 
largest number of cores is used.  

\begin{figure}[htbp]
	\centering
	\subfloat[Total time to solution $T (s)$]{
%
%
\definecolor{mycolor1}{rgb}{0.00000,0.44700,0.74100}%
\definecolor{mycolor2}{rgb}{0.85000,0.32500,0.09800}%
\definecolor{mycolor3}{rgb}{0.92900,0.69400,0.12500}%
\definecolor{mycolor4}{rgb}{0.49400,0.18400,0.55600}%
\begin{tikzpicture}

\begin{axis}[%
width=0.25\columnwidth,
height=0.25\columnwidth,
at={(0in,0in)},
scale only axis,
xmode=log,
xmin=1,
xmax=8192,
xminorticks=true,
xlabel style={font=\color{white!15!black}},
xlabel={$n_p$},
log basis x=2,
ymin=10,
ymax=80,
axis background/.style={fill=white},
legend style={at={(0.03,0.97)}, anchor=north west, legend cell align=left, align=left, draw=none, fill=none}
]
\addplot [color=mycolor1, mark=o, line width=2.0pt, mark options={solid, mycolor1}]
  table[row sep=crcr]{%
1	25.810033\\
4	31.604952\\
16	38.398774\\
64	39.91655\\
256	45.933118\\
1024	50.159127\\
4096	62.590067\\
8192	57.762285\\
};
\addlegendentry{\textsc{KIN}/VSMATCH}

\addplot [color=mycolor2, mark=x, line width=2.0pt, mark options={solid, mycolor2}]
  table[row sep=crcr]{%
1	18.279254\\
4	22.795919\\
16	26.43832\\
64	27.132647\\
256	31.282759\\
1024	37.394473\\
4096	42.908839\\
8192	49.020841\\
};
\addlegendentry{\textsc{KIN}/VDSVMB}

\addplot [color=mycolor3, mark=triangle, line width=2.0pt, mark options={solid, mycolor3}]
  table[row sep=crcr]{%
1	13.989477\\
4	18.746585\\
16	24.181467\\
64	25.3002\\
256	27.607395\\
1024	30.956618\\
4096	43.926954\\
8192	48.401245\\
};
\addlegendentry{\textsc{KIN}/AS}


\end{axis}
\end{tikzpicture}%
}\hfil
	\subfloat[Scaled Efficiency\label{fig:total-scaledeff}]{
%
%
\definecolor{mycolor1}{rgb}{0.00000,0.44700,0.74100}%
\definecolor{mycolor2}{rgb}{0.85000,0.32500,0.09800}%
\definecolor{mycolor3}{rgb}{0.92900,0.69400,0.12500}%
\definecolor{mycolor4}{rgb}{0.49400,0.18400,0.55600}%
\begin{tikzpicture}

\begin{axis}[%
width=0.25\columnwidth,
height=0.25\columnwidth,
at={(0in,0in)},
scale only axis,
xmode=log,
log basis x=2,
xlabel style={font=\color{white!15!black}},
xlabel={$n_p$},
ylabel={$E_p$},
xmin=1,
xmax=8192,
xminorticks=true,
ymin=25,
ymax=100,
axis background/.style={fill=white},
legend style={at={(0.37,0.97)}, anchor=north west, legend cell align=left, align=left, draw=none, fill=none}
]
\addplot [color=mycolor1, mark=o, line width=2.0pt, mark options={solid, mycolor1}]
  table[row sep=crcr]{%
1	100\\
4	81.6645220660357\\
16	67.215773607772\\
64	64.6599793819857\\
256	56.1904658856383\\
1024	51.4563042534612\\
4096	41.2366278502306\\
8192	44.6831925018202\\
};
\addlegendentry{\textsc{KIN}/VSMATCH}

\addplot [color=mycolor2, mark=x, line width=2.0pt, mark options={solid, mycolor2}]
  table[row sep=crcr]{%
1	100\\
4	80.1865193502398\\
16	69.1392418277712\\
64	67.3699620976899\\
256	58.432358859396\\
1024	48.8822345484051\\
4096	42.6002064516357\\
8192	37.2887401095383\\
};
\addlegendentry{\textsc{KIN}/VDSVMB}

\addplot [color=mycolor3, mark=triangle, line width=2.0pt, mark options={solid, mycolor3}]
  table[row sep=crcr]{%
1	100\\
4	74.6241355425535\\
16	57.8520608365076\\
64	55.2939383878388\\
256	50.6729338280559\\
1024	45.1905857416337\\
4096	31.8471364984697\\
8192	28.903134619781\\
};
\addlegendentry{\textsc{KIN}/AS}


\end{axis}
\end{tikzpicture}%
}
	\caption{Weak scaling: Total execution time and Scaled efficiency shown in percentage for the solution of the whole problem (all time steps), \textsc{Kinsol} with different linear solvers.}
	\label{fig:implementationscalability}
\end{figure}
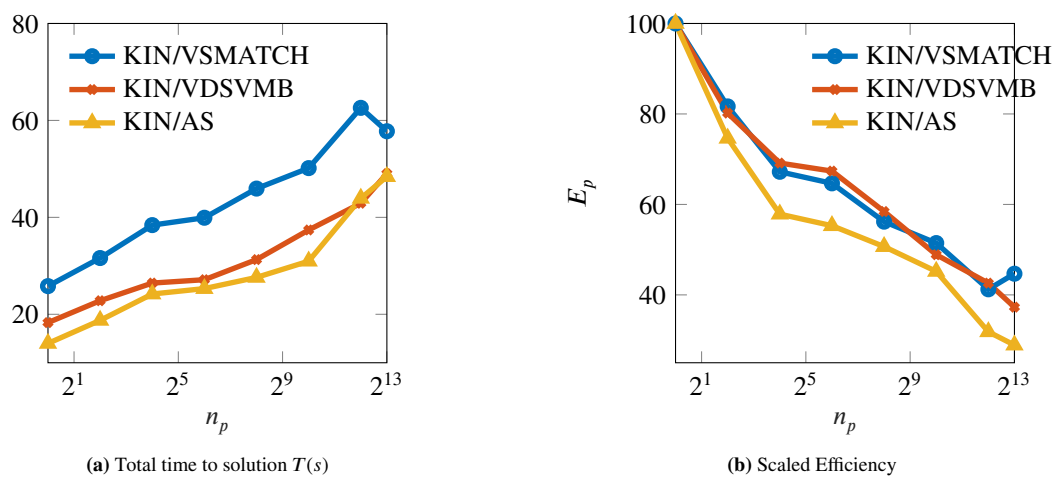

To give a complete picture of the performance of the various parts of the solution procedure we also 
report the percentage of the execution time spent in the different computational kernels. The barplots 
in Figure~\ref{fig:timespentpercentage} deliver several useful information.
\begin{figure}[htbp]
	\centering
	\subfloat[VDSVMB]{
%
%
\definecolor{mycolor1}{rgb}{0.00000,0.44700,0.74100}%
\definecolor{mycolor2}{rgb}{0.85000,0.32500,0.09800}%
\definecolor{mycolor3}{rgb}{0.92900,0.69400,0.12500}%
\definecolor{mycolor4}{rgb}{0.49400,0.18400,0.55600}%
\definecolor{mycolor5}{rgb}{0.46600,0.67400,0.18800}%
\definecolor{mycolor6}{rgb}{0.30100,0.74500,0.93300}%
\definecolor{mycolor7}{rgb}{0.63500,0.07800,0.18400}%
\begin{tikzpicture}

\begin{axis}[%
width=0.25\columnwidth,
height=0.25\columnwidth,
at={(0.758in,0.348in)},
scale only axis,
bar width=12,
xmin=-0.8,
xmax=14.8,
xtick={0,2,4,6,8,10,12,14},
xticklabels={ $2^0$,  $2^2$,  $2^4$,  $2^6$,  $2^8$, $2^{10}$, $2^{12}$, $2^{13}$},
ymin=0,
ymax=1,
legend columns = 3,
axis background/.style={fill=white},
legend style={at={(0.5,1.1)}, anchor=south, legend cell align=left, align=left, draw=none}
]
\addplot[ybar stacked, fill=mycolor1, draw=black, area legend] table[row sep=crcr] {%
0	8.40833001171711e-06\\
2	0.000189436846130222\\
4	0.000311732704649917\\
6	0.000494932470097739\\
8	0.00108125472564616\\
10	0.0014162273125229\\
12	0.00124785848901668\\
14	0.00117906736072452\\
};
\addplot[forget plot, color=white!15!black] table[row sep=crcr] {%
-0.4	0\\
15.4	0\\
};

\addplot[ybar stacked, fill=mycolor2, draw=black, area legend] table[row sep=crcr] {%
0	0.249967161132506\\
2	0.208600918874997\\
4	0.212455980561548\\
6	0.207162474785449\\
8	0.196135877912814\\
10	0.191092769244268\\
12	0.195785497715284\\
14	0.18991797467019\\
};
\addplot[forget plot, color=white!15!black] table[row sep=crcr] {%
-0.4	0\\
15.4	0\\
};

\addplot[ybar stacked, fill=mycolor3, draw=black, area legend] table[row sep=crcr] {%
0	0.0514195491785387\\
2	0.0429805001500488\\
4	0.0468626977810996\\
6	0.0428954093568534\\
8	0.0383276615723057\\
10	0.0331911884411367\\
12	0.0300287313763022\\
14	0.034545570142084\\
};
\addplot[forget plot, color=white!15!black] table[row sep=crcr] {%
-0.4	0\\
15.4	0\\
};

\addplot[ybar stacked, fill=mycolor4, draw=black, area legend] table[row sep=crcr] {%
0	0.0467053524175549\\
2	0.0393409451928654\\
4	0.0433329727456208\\
6	0.0372574043365544\\
8	0.0360266177289542\\
10	0.0301908787429629\\
12	0.026875278541095\\
14	0.0342358914892665\\
};
\addplot[forget plot, color=white!15!black] table[row sep=crcr] {%
-0.4	0\\
15.4	0\\
};

\addplot[ybar stacked, fill=mycolor5, draw=black, area legend] table[row sep=crcr] {%
0	0.0216051929690347\\
2	0.0202683607535191\\
4	0.0213389542149426\\
6	0.0230699938712209\\
8	0.021651105006435\\
10	0.0219093233377029\\
12	0.0204203171285991\\
14	0.0247051489086589\\
};
\addplot[forget plot, color=white!15!black] table[row sep=crcr] {%
-0.4	0\\
15.4	0\\
};

\addplot[ybar stacked, fill=mycolor6, draw=black, area legend] table[row sep=crcr] {%
0	0.621418759759014\\
2	0.680088546112135\\
4	0.66586515482073\\
6	0.677880690372745\\
8	0.692513456373845\\
10	0.706255614298937\\
12	0.705292521897411\\
14	0.696422780723401\\
};
\addplot[forget plot, color=white!15!black] table[row sep=crcr] {%
-0.4	0\\
15.4	0\\
};

\addplot[ybar stacked, fill=mycolor7, draw=black, area legend] table[row sep=crcr] {%
0	0.00887557621333977\\
2	0.00853129207030449\\
4	0.00983250717140874\\
6	0.0112390948070789\\
8	0.0142640266799998\\
10	0.0159439986224699\\
12	0.0203497948522912\\
14	0.0189935667056743\\
};
\addplot[forget plot, color=white!15!black] table[row sep=crcr] {%
-0.4	0\\
15.4	0\\
};

\end{axis}
\end{tikzpicture}%
}
	\subfloat[VSMATCH]{
%
%
\definecolor{mycolor1}{rgb}{0.00000,0.44700,0.74100}%
\definecolor{mycolor2}{rgb}{0.85000,0.32500,0.09800}%
\definecolor{mycolor3}{rgb}{0.92900,0.69400,0.12500}%
\definecolor{mycolor4}{rgb}{0.49400,0.18400,0.55600}%
\definecolor{mycolor5}{rgb}{0.46600,0.67400,0.18800}%
\definecolor{mycolor6}{rgb}{0.30100,0.74500,0.93300}%
\definecolor{mycolor7}{rgb}{0.63500,0.07800,0.18400}%
\begin{tikzpicture}

\begin{axis}[%
width=0.25\columnwidth,
height=0.25\columnwidth,
at={(0.758in,0.348in)},
scale only axis,
bar width=12,
xmin=-0.8,
xmax=14.8,
xtick={0,2,4,6,8,10,12,14},
xticklabels={ $2^0$,  $2^2$,  $2^4$,  $2^6$,  $2^8$, $2^{10}$, $2^{12}$, $2^{13}$},
ymin=0,
ymax=1,
legend columns = 3,
axis background/.style={fill=white},
legend style={at={(0.5,1.1)}, anchor=south, legend cell align=left, align=left, draw=none}
]
\addplot[ybar stacked, fill=mycolor1, draw=black, area legend] table[row sep=crcr] {%
0	6.24388973078802e-06\\
2	0.000129604531593657\\
4	0.000222757164069874\\
6	0.000292925991850498\\
8	0.000462202737444202\\
10	0.000556133124087267\\
12	0.000588324917434583\\
14	0.00102353620532844\\
};
\addplot[forget plot, color=white!15!black] table[row sep=crcr] {%
-0.4	0\\
15.4	0\\
};

\addplot[ybar stacked, fill=mycolor2, draw=black, area legend] table[row sep=crcr] {%
0	0.169403489333005\\
2	0.146921956090932\\
4	0.144946961067038\\
6	0.140822082569761\\
8	0.141737590811387\\
10	0.135043550498796\\
12	0.141527611721521\\
14	0.138449335929145\\
};
\addplot[forget plot, color=white!15!black] table[row sep=crcr] {%
-0.4	0\\
15.4	0\\
};

\addplot[ybar stacked, fill=mycolor3, draw=black, area legend] table[row sep=crcr] {%
0	0.0361348627489163\\
2	0.0299030987295915\\
4	0.028770553976541\\
6	0.0283101119710997\\
8	0.0267109219140013\\
10	0.024917578808738\\
12	0.0283304697532917\\
14	0.0241178651433197\\
};
\addplot[forget plot, color=white!15!black] table[row sep=crcr] {%
-0.4	0\\
15.4	0\\
};

\addplot[ybar stacked, fill=mycolor4, draw=black, area legend] table[row sep=crcr] {%
0	0.0326647393283069\\
2	0.0274529763563634\\
4	0.0255207106351885\\
6	0.0260073077457846\\
8	0.0250653620970203\\
10	0.0228381167798235\\
12	0.0265830998391486\\
14	0.0231846264392068\\
};
\addplot[forget plot, color=white!15!black] table[row sep=crcr] {%
-0.4	0\\
15.4	0\\
};

\addplot[ybar stacked, fill=mycolor5, draw=black, area legend] table[row sep=crcr] {%
0	0.0211171988815357\\
2	0.0253064266637709\\
4	0.027538814129847\\
6	0.0300484510810679\\
8	0.0297963528109837\\
10	0.0337277042321729\\
12	0.0328658683174121\\
14	0.0362212505963017\\
};
\addplot[forget plot, color=white!15!black] table[row sep=crcr] {%
-0.4	0\\
15.4	0\\
};

\addplot[ybar stacked, fill=mycolor6, draw=black, area legend] table[row sep=crcr] {%
0	0.734557294444374\\
2	0.764309860998998\\
4	0.765893208465458\\
6	0.766862156173316\\
8	0.766425937974655\\
10	0.770233206012537\\
12	0.756401428680369\\
14	0.760565862309637\\
};
\addplot[forget plot, color=white!15!black] table[row sep=crcr] {%
-0.4	0\\
15.4	0\\
};

\addplot[ybar stacked, fill=mycolor7, draw=black, area legend] table[row sep=crcr] {%
0	0.0061161713741319\\
2	0.00597607662875116\\
4	0.00710699456185773\\
6	0.00765696446711961\\
8	0.00980163165450851\\
10	0.0126837105438457\\
12	0.0137031967708227\\
14	0.0164375233770616\\
};
\addplot[forget plot, color=white!15!black] table[row sep=crcr] {%
-0.4	0\\
15.4	0\\
};

\end{axis}
\end{tikzpicture}%
}
	\subfloat[AS]{
%
%
\definecolor{mycolor1}{rgb}{0.00000,0.44700,0.74100}%
\definecolor{mycolor2}{rgb}{0.85000,0.32500,0.09800}%
\definecolor{mycolor3}{rgb}{0.92900,0.69400,0.12500}%
\definecolor{mycolor4}{rgb}{0.49400,0.18400,0.55600}%
\definecolor{mycolor5}{rgb}{0.46600,0.67400,0.18800}%
\definecolor{mycolor6}{rgb}{0.30100,0.74500,0.93300}%
\definecolor{mycolor7}{rgb}{0.63500,0.07800,0.18400}%
\begin{tikzpicture}

\begin{axis}[%
width=0.25\columnwidth,
height=0.25\columnwidth,
at={(0.758in,0.348in)},
scale only axis,
bar width=12,
xmin=-0.8,
xmax=14.8,
xtick={0,2,4,6,8,10,12,14},
xticklabels={ $2^0$,  $2^2$,  $2^4$,  $2^6$,  $2^8$, $2^{10}$, $2^{12}$,$2^{13}$},
ymin=0,
ymax=1,
axis background/.style={fill=white},
legend style={at={(1.25,0.1)}, anchor=south, legend cell align=left, align=left, draw=none, fill=none}
]
\addplot[ybar stacked, fill=mycolor1, draw=black, area legend] table[row sep=crcr] {%
0	1.17138045975557e-05\\
2	0.000201716312597734\\
4	0.000356392728365074\\
6	0.00045530248772737\\
8	0.000830661639752682\\
10	0.00114878808789772\\
12	0.00139945892902112\\
14	0.00153374482412591\\
};
\addplot[forget plot, color=white!15!black] table[row sep=crcr] {%
-0.4	0\\
15.4	0\\
};
\addlegendentry{Halo}

\addplot[ybar stacked, fill=mycolor2, draw=black, area legend] table[row sep=crcr] {%
0	0.33901319970718\\
2	0.237917470835355\\
4	0.228738091034758\\
6	0.228243227326266\\
8	0.225720054355002\\
10	0.225560579647299\\
12	0.199556502369821\\
14	0.195205579939111\\
};
\addplot[forget plot, color=white!15!black] table[row sep=crcr] {%
-0.4	0\\
15.4	0\\
};
\addlegendentry{Feval}

\addplot[ybar stacked, fill=mycolor3, draw=black, area legend] table[row sep=crcr] {%
0	0.0659066811432622\\
2	0.0505794522042281\\
4	0.0455652669873172\\
6	0.0598597244290559\\
8	0.0557113773320518\\
10	0.0515149296993619\\
12	0.0415660507669164\\
14	0.0493687920630967\\
};
\addplot[forget plot, color=white!15!black] table[row sep=crcr] {%
-0.4	0\\
15.4	0\\
};
\addlegendentry{Jacobian}

\addplot[ybar stacked, fill=mycolor4, draw=black, area legend] table[row sep=crcr] {%
0	0.0595308888245072\\
2	0.046363964423387\\
4	0.0398650751834039\\
6	0.0544281863384479\\
8	0.052041418612658\\
10	0.0472894681195472\\
12	0.0367687912073302\\
14	0.0459636317206303\\
};
\addplot[forget plot, color=white!15!black] table[row sep=crcr] {%
-0.4	0\\
15.4	0\\
};
\addlegendentry{Auxiliary}

\addplot[ybar stacked, fill=mycolor5, draw=black, area legend] table[row sep=crcr] {%
0	0.0315767916127243\\
2	0.0277588958202254\\
4	0.0277679182987533\\
6	0.029963616888404\\
8	0.0376751700042688\\
10	0.0371249566086321\\
12	0.0288236967216074\\
14	0.0321728232403939\\
};
\addplot[forget plot, color=white!15!black] table[row sep=crcr] {%
-0.4	0\\
15.4	0\\
};
\addlegendentry{Setup}

\addplot[ybar stacked, fill=mycolor6, draw=black, area legend] table[row sep=crcr] {%
0	0.491699129281245\\
2	0.627630421220718\\
4	0.646950683761246\\
6	0.614672008126418\\
8	0.614393574257912\\
10	0.616917777000059\\
12	0.674646998970154\\
14	0.659014593736173\\
};
\addplot[forget plot, color=white!15!black] table[row sep=crcr] {%
-0.4	0\\
15.4	0\\
};
\addlegendentry{LinSol}

\addplot[ybar stacked, fill=mycolor7, draw=black, area legend] table[row sep=crcr] {%
0	0.0122615956264842\\
2	0.00954807918348843\\
4	0.0107565720061568\\
6	0.0123779344036808\\
8	0.0136277437983554\\
10	0.0204435008372038\\
12	0.0172385010351505\\
14	0.0167408344764685\\
};
\addplot[forget plot, color=white!15!black] table[row sep=crcr] {%
-0.4	0\\
15.4	0\\
};
\addlegendentry{Overhead}

\end{axis}
\end{tikzpicture}%
}
	\caption{Percentage of the overall time spent in the different phases for the overall solution process. We sum all the partial times and normalize it against the total to account for the different computational kernels.}
	\label{fig:timespentpercentage}
\end{figure}
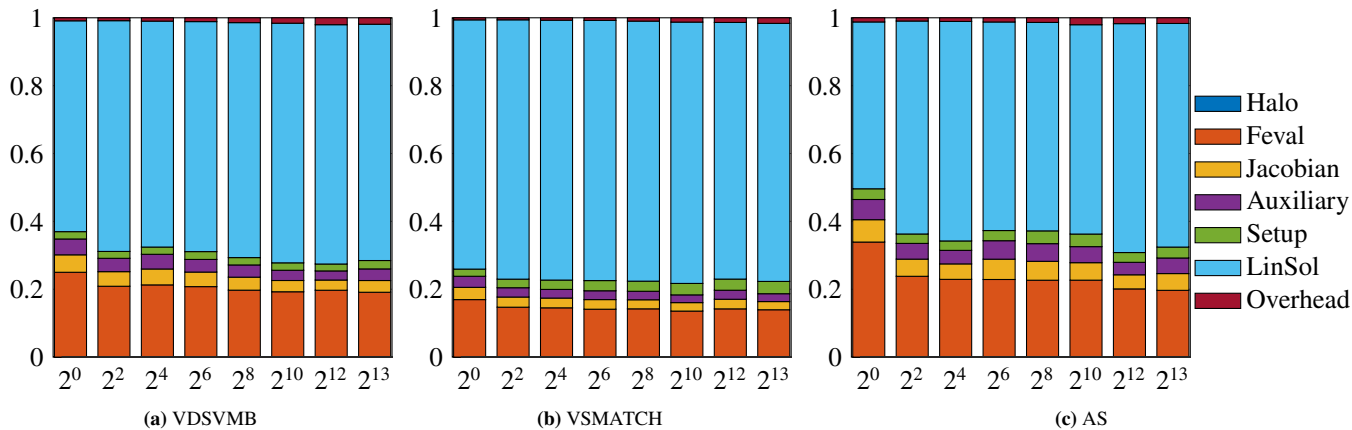
First of all, as expected, the two main parts of the efforts are
represented by the time spent in actually solving the linear
systems~\eqref{eq:linearsystemtosolve} (LinSol) and evaluating the
non-linear function $\boldsymbol{\Phi}$ encoding the discretization
(Feval), on which our implementation efforts were more focused. In
particular, in the Feval kernel we exploited many of the
support functionalities of PSCToolkit. To increase the efficiency, at the
expense of a marginal increase in the use of memory, we experimented with
memorizing the \emph{local-to-global} map on the MPI tasks, thus avoiding
the overhead of explicit calls to the functions that permit the user to 
explore the index space. This is made possible by the fact that we are
using a communicator for data distribution which is fixed through the time
steps and iterates of Newton's method; see the discussion in
Section~\ref{sec:data-distribution}. Thanks to the flexibility of the
underlying software framework, we were able to optimize the procedure
for building the vectors and matrices at each nonlinear time step, 
making good reuse of local data.

We have a small and almost equal time that is spent in building the
distributed Jacobian (Jacobian) and auxiliary matrices (Auxiliary). The
rebuild of the local portions of the matrices is well optimized and is
very close to the absolute minimum that is necessary to compute the
coefficient values.
A reasonably small amount of time is spent in building and updating the
preconditioners (Setup). In particular, it is possible to reuse and apply
the aggregation hierarchy, so that we only need to recompute the
smoothers. 
The other (almost invisible) bar is represented
by the halo exchanges: this is the data exchange among the parallel
processes that happens before each build of the Jacobian, build of auxiliary
matrices and function evaluation. It is a communication that 
is necessary for having an agreement on all the quantities needed by the
processes to perform their computations, namely the exchange of the
boundary data; this operation is persistent, in the sense that the pattern of the data exchange 
is determined by the discretization mesh structure.
The last overhead is all the remaining computations that
are made inside the Newton method, i.e., vector updates and vector norm
computations, namely the Newton direction updates, convergence checks and
line-searches, and are in charge to the KINSOL library.

\section{Conclusions and perspectives}
\label{sec:conclusions}

In this paper we focused on two main objectives: to prove some spectral
properties of the sequence of Jacobian matrices generated discretizing 
the Richards equation in mixed form for simulation of unsaturated
subsurface flows by a Newton-type method, and to prove the
efficiency, flexibility and robustness of a software framework for
parallel sparse matrix computations.

The theoretical results we obtained are consistent
with expectations and justify some preconditioner choices already used in
the literature for solving Richards equations. 
We used several functionalities of the PSCToolkit for iterative
solution of sparse linear systems, some of the support routines for
re-using preconditioners in solving sequences of linear
systems and for an efficient setup and update of data structures needed
for Jacobian and right-hand side computations in Newton iterations. The
performances of our strategy on one of the
most powerful supercomputers currently available are quite promising in view of exploring extreme scalability and confirm the benefits of using multigrid
preconditioners when number of processing cores largely increases.

Our plans for future work include the extension of the PSCToolkit
interface to KINSOL, in order to use the ability of the PSCToolkit linear
solvers in exploiting GPU architectures, and the integration of the
software stack into the Parflow code~\cite{PARFLOW-MANUAL} for realistic
simulations in hydrological applications.

\section*{Acknowledgments}
We thank Prof. Stefan Kollet from the J\"{u}lich Supercomputing Center for the information regarding the
implementation of the Richards equation in \textsc{ParFlow} and the insightful discussions regarding its
discretization.
We acknowledge PRACE (Partnership for Advanced Computing in Europe) for awarding us access to the high-performance computing system Marconi100 at CINECA (Consorzio Interuniversitario per il Calcolo Automatico dell’Italia Nord-orientale) under the grants Pra22$\_$EoCoE, Pra23$\_$EoCoE and Pra24$\_$EoCoE.

\subsection*{Financial disclosure}

This work has been partially supported by EU under the Horizon 2020 Project "Energy oriented Centre of
Excellence: toward exascale for energy (EoCoE-II), Project ID: 824158.

\subsection*{Conflict of interest}

The authors declare no potential conflict of interests.

\bibliography{richardsbibliography.bib}%

\appendix

\section{Supplementary materials}

Here we gather further details regarding both theoretical analysis and numerical experiments. This material can be used to illustrate in greater detail the analyzes described in the paper.

We start by reporting here the expression of the the Jacobian matrix for arithmetic average~\ref{sec:jacobian_arithmetic}. Then we report the
theoretical analysis for the upstream average formulation in Section~\ref{sec:upstream_average}, this complements the analysis made for the arithmetic average in Section~\ref{sec:arithmetic_average} and, a part of some technical refinements, permits to state and conclude the analysis concerning the spectrum in Theorem~\ref{thm:spectraldescription}. For the sake of completion we also report the whole closed formulation of the 3D Jacobian in Section~\ref{sec:jacobian-expression}. To validate the spectral analysis also in a numerical way we show in Section~\ref{theoreticalresults} the comparison between the spectral distribution and the computed eigenvalues for a 1D problem and also some weak scalability results to compare the behaviour of the AS preconditioner, also applicable to unsymmetric matrices, on the original Jabobians sequence with respect to the spectrally equivalent sequence.
Section~\ref{sec:algscalability} discusses the algorithmic scalability when the theory is applied as is. That is, we show that the multigrid preconditioners built on the spectral approximation gives a number of iteration that remains almost independent of the discretization size.

Finally, Section~\ref{sec:details-on-build-and-solve} gives further details on the numerical results for the weak scaling analysis in Section~\ref{sec:weak_scalability}.

\section{The Jacobian matrix for arithmetic average}\label{sec:jacobian_arithmetic}

We give here the complete expression of the Jacobian with respect to the arithmetic average.

\begingroup
\allowdisplaybreaks
\begin{align*}
\frac{\partial \boldsymbol{\Phi}}{\partial p_{i,j,k}^{(l)}} = &  \frac{\rho  \phi  s'\left(p_{i,j,k}^{(l)}\right)}{\Delta t} + \frac{K\left[p_{i-1,j,k}^{(l)}\right]+K\left[p_{i+1,j,k}^{(l)}\right]+2 K\left[p_{i,j,k}^{(l)}\right]}{2 h_x^2} + \frac{K\left[p_{i,j-1,k}^{(l)}\right]+K\left[p_{i,j+1,k}^{(l)}\right]+2 K\left[p_{i,j,k}^{(l)}\right]}{2 h_y^2} \\
& + \frac{K\left[p_{i,j,k-1}^{(l)}\right]+K\left[p_{i,j,k+1}^{(l)}\right]+2 K\left[p_{i,j,k}^{(l)}\right]}{2 h_z^2} +p_{i,j,k}^{(l)} \left(\frac{K'\left(p_{i,j,k}^{(l)}\right)}{h_x^2}+\frac{K'\left(p_{i,j,k}^{(l)}\right)}{h_y^2}+\frac{K'\left(p_{i,j,k}^{(l)}\right)}{h_z^2}\right) \\
& -\frac{p_{i,j-1,k}^{(l)} K'\left[p_{i,j,k}^{(l)}\right]}{2 h_y^2}-\frac{p_{i,j,k-1}^{(l)} K'\left[p_{i,j,k}^{(l)}\right]}{2 h_z^2}-\frac{p_{i,j,k+1}^{(l)} K'\left[p_{i,j,k}^{(l)}\right]}{2 h_z^2} \\
& -\frac{p_{i-1,j,k}^{(l)} K'\left[p_{i,j,k}^{(l)}\right]}{2 h_x^2}-\frac{p_{i+1,j,k}^{(l)} K'\left[p_{i,j,k}^{(l)}\right]}{2 h_x^2}-\frac{p_{i,j+1,k}^{(l)} K'\left[p_{i,j,k}^{(l)}\right]}{2 h_y^2} ,\stepcounter{equation}\\
\frac{\partial \boldsymbol{\Phi}}{\partial p_{i-1,j,k}^{(l)}} = & -\frac{K\left[p_{i-1,j,k}^{(l)}\right]+K\left[p_{i,j,k}^{(l)}\right]}{2 h_x^2}-\frac{p_{i-1,j,k}^{(l)} K'\left[p_{i-1,j,k}^{(l)}\right]}{2 h_x^2} +\frac{p_{i,j,k}^{(l)} K'\left[p_{i-1,j,k}^{(l)}\right]}{2 h_x^2},\tag{\theequation}\label{eq:jacobian_arithmeticaverage} \\
\frac{\partial \boldsymbol{\Phi}}{\partial p_{i+1,j,k}^{(l)}} = & -\frac{K\left[p_{i+1,j,k}^{(l)}\right]+K\left[p_{i,j,k}^{(l)}\right]}{2 h_x^2}-\frac{p_{i+1,j,k}^{(l)} K'\left[p_{i+1,j,k}^{(l)}\right]}{2 h_x^2} +\frac{p_{i,j,k}^{(l)} K'\left[p_{i+1,j,k}^{(l)}\right]}{2 h_x^2}, \\
\frac{\partial \boldsymbol{\Phi}}{\partial p_{i,j-1,k}^{(l)}} = & -\frac{K\left[p_{i,j-1,k}^{(l)}\right]+K\left[p_{i,j,k}^{(l)}\right]}{2 h_y^2}-\frac{p_{i,j-1,k}^{(l)} K'\left[p_{i,j-1,k}^{(l)}\right]}{2 h_y^2} +\frac{p_{i,j,k}^{(l)} K'\left[p_{i,j-1,k}^{(l)}\right]}{2 h_y^2},\\
\frac{\partial \boldsymbol{\Phi}}{\partial p_{i,j+1,k}^{(l)}} = & -\frac{K\left[p_{i,j+1,k}^{(l)}\right]+K\left[p_{i,j,k}^{(l)}\right]}{2 h_y^2}-\frac{p_{i,j+1,k}^{(l)} K'\left[p_{i,j+1,k}^{(l)}\right]}{2 h_y^2} +\frac{p_{i,j,k}^{(l)} K'\left[p_{i,j+1,k}^{(l)}\right]}{2 h_y^2}, \\
\frac{\partial \boldsymbol{\Phi}}{\partial p_{i,j,k-1}^{(l)}} = & -\frac{K\left[p_{i,j,k-1}^{(l)}\right]+K\left[p_{i,j,k}^{(l)}\right]-h_z K'\left[p_{i,j,k-1}^{(l)}\right]}{2 h_z^2} -\frac{p_{i,j,k-1}^{(l)} K'\left[p_{i,j,k-1}^{(l)}\right]}{2 h_z^2}+\frac{p_{i,j,k}^{(l)} K'\left[p_{i,j,k-1}^{(l)}\right]}{2 h_z^2}, \\
\frac{\partial \boldsymbol{\Phi}}{\partial p_{i,j,k+1}^{(l)}} = & -\frac{K\left[p_{i,j,k+1}^{(l)}\right]+K\left[p_{i,j,k}^{(l)}\right]+h_z K'\left[p_{i,j,k+1}^{(l)}\right]}{2 h_z^2}  -\frac{p_{i,j,k+1}^{(l)} K'\left[p_{i,j,k+1}^{(l)}\right]}{2 h_z^2}+\frac{p_{i,j,k}^{(l)} K'\left[p_{i,j,k+1}^{(l)}\right]}{2 h_z^2}.
\end{align*}
\endgroup

\section{Spectral analysis for the upstream average formulation}
\label{sec:upstream_average}

We consider here the case in which $\prescript{\text{\scriptsize AV}}{}{K_{L,U}}$ is given by the upstream-weighted mean of the values of $K$ in~\eqref{eq:K(p)} at the two side of the interface. The machinery we use to obtain the result for this case is mostly identical to the arithmetic average case discussed in Section~\ref{sec:arithmetic_average}. There is only a technical complication given by the form of the coefficients,
\begin{equation*}
\begin{split}
\prescript{\text{UP}}{}{K^{(l+1)}_{i+1,i}} =  \begin{cases}
K[p^{(l+1)}_{i+1}], & p^{(l+1)}_{i+1} -  p^{(l+1)}_{i} \geq 0,\\
K[p^{(l+1)}_{i}], & p^{(l+1)}_{i+1} -  p^{(l+1)}_{i} < 0,\\
\end{cases} \qquad
\prescript{\text{UP}}{}{K^{(l+1)}_{i,i-1}} =  \begin{cases}
K[p^{(l+1)}_{i}], & p^{(l+1)}_{i} -  p^{(l+1)}_{i-1} \geq 0,\\
K[p^{(l+1)}_{i-1}], & p^{(l+1)}_{i} -  p^{(l+1)}_{i-1} < 0.\\
\end{cases}
\end{split}
\end{equation*}
With this choice equation~\eqref{eq:dynamic_of_the_problem} becomes
\begin{equation*}
\begin{split}
\prescript{{\text{\tiny UP}}}{}{\Phi}(p^{(l+1)}_i) = & \frac{S\left(p_i^{(l+1)}\right)-S\left(p_i^{(l)}\right)}{\Delta t}  -\frac{1}{{h_z}^2}\left[  \left(p_{i+1}^{(l+1)}-p_i^{(l+1)}\right)\begin{cases}
K[p^{(l+1)}_{i+1}], & p^{(l+1)}_{i+1} -  p^{(l+1)}_{i} \geq 0,\\
K[p^{(l+1)}_{i}], & p^{(l+1)}_{i+1} -  p^{(l+1)}_{i} < 0,\\
\end{cases}\right.\\&\left.-\left(p_i^{(l+1)}-p_{i-1}^{(l+1)}\right)\begin{cases}
K[p^{(l+1)}_{i}], & p^{(l+1)}_{i} -  p^{(l+1)}_{i-1} \geq 0,\\
K[p^{(l+1)}_{i-1}], & p^{(l+1)}_{i} -  p^{(l+1)}_{i-1} < 0.\\
\end{cases}\right]  -\frac{1}{2 h_z}\left( K\left[p_{i+1}^{(l+1)}\right]-K\left[p_{i-1}^{(l+1)}\right] \right).
\end{split}
\end{equation*}
The Jacobian matrix is then again a tridiagonal matrix with entries
\begin{equation}\label{eq:jacobian_upstream_mean}
\prescript{{\text{\tiny UP}}}{}{J_N} = \operatorname{tridiag}(\eta_i,\zeta_i,\xi_i) = \begin{bmatrix}
\zeta_1 & \xi_1 \\
\eta_1 & \ddots & \ddots \\
& \ddots &  \ddots & \xi_{N-3} \\
& & \eta_{N-3} & \zeta_{N-2}
\end{bmatrix},
\end{equation}
where,
\begingroup
\allowdisplaybreaks
\begin{align*}
\zeta_i = & \frac{s'\left(p_i^{(l+1)}\right)}{\Delta t} - \frac{p_{i-1}^{(l+1)}}{h_z^2} \begin{cases}
K'\left[p_i^{(l+1)}\right], & p_i^{(l+1)}-p_{i-1}^{(l+1)} \geq 0 \\
0, & p_i^{(l+1)}-p_{i-1}^{(l+1)} < 0
\end{cases}   - \frac{p_{i+1}^{(l+1)}}{h_z^2} \begin{cases}
0, & p_{i+1}^{(l+1)}-p_i^{(l+1)} \geq 0\\
K'\left[p_i^{(l+1)}\right], & p_{i+1}^{(l+1)}-p_i^{(l+1)} < 0 \\
\end{cases} \\ & + \frac{p_i^{(l+1)}}{h_z^2} \left(\begin{cases}
K'\left[p_i^{(l+1)}\right], & p_i^{(l+1)}-p_{i-1}^{(l+1)} \geq 0\\
0, & p_i^{(l+1)}-p_{i-1}^{(l+1)} < 0\\
\end{cases} + \begin{cases}
0, & p_{i+1}^{(l+1)}-p_i^{(l+1)} \geq 0\\
K'\left[p_i^{(l+1)}\right], & p_{i+1}^{(l+1)}-p_i^{(l+1)}< 0
\end{cases}\right) \\
& + \frac{1}{h_z^2} \left(\begin{cases}
K\left[p_i^{(l+1)}\right], & p_i^{(l+1)}-p_{i-1}^{(l+1)} \geq 0 \\
K\left[p_{i-1}^{(l+1)}\right], & p_i^{(l+1)}-p_{i-1}^{(l+1)} < 0
\end{cases}  + \begin{cases}
K\left[p_{i+1}^{(l+1)}\right], & p_{i+1}^{(l+1)}-p_i^{(l+1)} \geq 0 \\
K\left[p_i^{(l+1)}\right], & p_{i+1}^{(l+1)}-p_i^{(l+1)} < 0
\end{cases} \right), \\
\xi_i = & -\frac{K'\left(p_{i+1}^{(l+1)}\right)}{2 h_z}  - \frac{1}{h_z^2} \begin{cases}
K\left[p_{i+1}^{(l+1)}\right],& p_{i+1}^{(l+1)}-p_i^{(l+1)}\geq 0\\
K\left[p_i^{(l+1)}\right], & p_{i+1}^{(l+1)}-p_i^{(l+1)} < 0
\end{cases} + \frac{p_i^{(l+1)}}{h_z^2} \begin{cases}
K'\left[p_{i+1}^{(l+1)}\right], & p_{i+1}^{(l+1)}-p_i^{(l+1)} \geq 0\\
0, & p_{i+1}^{(l+1)}-p_i^{(l+1)} < 0
\end{cases} \\
& - \frac{p_{i+1}^{(l+1)}}{h_z^2} \begin{cases}
K'\left[p_{i+1}^{(l+1)}\right], & p_{i+1}^{(l+1)}-p_i^{(l+1)} \geq 0 \\
0, & p_{i+1}^{(l+1)}-p_i^{(l+1)}< 0
\end{cases},
\\
\eta_i = & \frac{K'\left(p_{i-1}^{(l+1)}\right)}{2 h_z} - \frac{1}{h_z^2} \begin{cases} K\left[p_i^{(l+1)}\right], & p_i^{(l+1)}-p_{i-1}^{(l+1)} \geq 0\\
K\left[p_{i-1}^{(l+1)}\right], & p_i^{(l+1)}-p_{i-1}^{(l+1)} < 0
\end{cases} + \frac{p_i^{(l+1)}}{h_z^2} \begin{cases}
0, & p_i^{(l+1)}-p_{i-1}^{(l+1)} \geq 0\\
K'\left[p_{i-1}^{(l+1)}\right] & p_i^{(l+1)}-p_{i-1}^{(l+1)} < 0
\end{cases} \\
& - \frac{p_{i-1}^{(l+1)}}{h_z^2} \begin{cases}
0, & p_i^{(l+1)}-p_{i-1}^{(l+1)} \geq 0 \\
K'\left[p_{i-1}^{(l+1)}\right], & p_i^{(l+1)}-p_{i-1}^{(l+1)} < 0\\
\end{cases}.
\end{align*}
\endgroup
To perform the spectral analysis we need to first rewrite the Jacobians as the sum of matrices that are more amenable to be treated, and use then the $*$-algebra and perturbation properties to obtain the spectral information on the different Jacobian expressions. As for the previous case, the three parts are one relative to the time-stepping, one relative to the Darcian diffusion, and the last one relative to the transport along the $z$-axis, i.e.,
\begin{equation*}
\{ \prescript{{\text{\tiny UP}}}{}{J_N} \}_N = \{ \prescript{{\text{\tiny UP}}}{}{D_N} \}_{N} + \{ \prescript{{\text{\tiny UP}}}{}{L_N} \}_N + \{ \prescript{{\text{\tiny UP}}}{}{T_N} \},
\end{equation*}
where $\{ \prescript{{\text{\tiny UP}}}{}{D_N} \}_{N}$ is the scaled
diagonal matrix sampling the function $s'(p)$ on $p$ approximated at
the $(i+1)$th time step that is unchanged from
$\prescript{{\text{\tiny ARIT}}}{}{D_N}$
in~\eqref{eq:arithmeticdiagonalsampling}. This means that $\{
\prescript{{\text{\tiny UP}}}{}{D_N} \}_{N} =
\prescript{{\text{\tiny ARIT}}}{}{D_N}$ has a spectral distribution
given again by Lemma~\ref{lem:symboltime}. In the same way, the
matrix $\prescript{{\text{\tiny UP}}}{}{T_N}$ is unchanged from
$\prescript{{\text{\tiny ARIT}}}{}{T_N}$
in~\eqref{eq:arithmetictransportterm}, i.e.,
$\prescript{{\text{\tiny UP}}}{}{T_N} = \prescript{{\text{\tiny
			ARIT}}}{}{T_N}$, whose spectra distribution is given in
Lemma~\ref{lem:symboltransport}. The only part for which we
encounter a difference is the one relative to the Darcian diffusion.
To shorten the notation and make the subsequent spectra proof easier
we pose
\begin{equation*}
{\Delta p}_{i}^{i+1} = p_{i+1}^{(l+1)}-p_i^{(l+1)}, \text{ and } {\Delta p}_{i-1}^{i} = p_i^{(l+1)}-p_{i-1}^{(l+1)},
\end{equation*}
then the decomposition reads as
\begingroup
\allowdisplaybreaks
\begin{align*}
\prescript{{\text{\tiny UP}}}{}{L_N} = & \frac{1}{h_z^2} \operatorname{tridiag}\left(- \begin{cases} K\left[p_i^{(l+1)}\right], & {\Delta p}_{i-1}^{i} \geq 0\\
K\left[p_{i-1}^{(l+1)}\right], & {\Delta p}_{i-1}^{i} < 0
\end{cases},\right. \\ & \left. \left( \begin{cases}
K\left[p_{i+1}^{(l+1)}\right], & {\Delta p}_{i}^{i+1} \geq 0 \\
K\left[p_i^{(l+1)}\right], & {\Delta p}_{i}^{i+1} < 0
\end{cases} +  \begin{cases}
K\left[p_i^{(l+1)}\right], & {\Delta p}_{i-1}^{i} \geq 0 \\
K\left[p_{i-1}^{(l+1)}\right], & {\Delta p}_{i-1}^{i} < 0
\end{cases} \right), \right. \\ &\left. -\begin{cases} K\left[p_i^{(l+1)}\right], & {\Delta p}_{i-1}^{i} \geq 0\\
K\left[p_{i-1}^{(l+1)}\right], & {\Delta p}_{i-1}^{i} < 0
\end{cases} \right) \\ & + \frac{1}{h_z^2} \operatorname{tridiag}\left( - p_{i-1}^{(l+1)} \begin{cases}
0, & {\Delta p}_{i-1}^{i} \geq 0 \\
K'\left[p_{i-1}^{(l+1)}\right], & {\Delta p}_{i-1}^{i} < 0\\
\end{cases} ,\right. \\ & \left. p_i^{(l+1)} \left(\begin{cases}
K'\left[p_i^{(l+1)}\right], & {\Delta p}_{i-1}^{i} \geq 0\\
0, & {\Delta p}_{i-1}^{i} < 0\\
\end{cases} + \begin{cases}
0, & {\Delta p}_{i}^{i+1} \geq 0\\
K'\left[p_i^{(l+1)}\right], & {\Delta p}_{i}^{i+1}< 0
\end{cases}\right), \right.\\&\left. - p_{i+1}^{(l+1)}\begin{cases}
K'\left[p_{i+1}^{(l+1)}\right], & {\Delta p}_{i}^{i+1} \geq 0 \\
0, & {\Delta p}_{i}^{i+1}< 0
\end{cases} \right) \\
& -\frac{1}{h^2_z} \operatorname{tridiag}\left( - p_i^{(l+1)} \begin{cases}
0, & {\Delta p}_{i-1}^{i} \geq 0\\
K'\left[p_{i-1}^{(l+1)}\right] & {\Delta p}_{i-1}^{i} < 0
\end{cases} , \right.\\&\left. p_{i-1}^{(l+1)} \begin{cases}
K'\left[p_i^{(l+1)}\right], & {\Delta p}_{i-1}^{i} \geq 0 \\
0, & {\Delta p}_{i-1}^{i} < 0
\end{cases}  + p_{i+1}^{(l+1)} \begin{cases}
0, & {\Delta p}_{i}^{i+1} \geq 0\\
K'\left[p_i^{(l+1)}\right], & {\Delta p}_{i}^{i+1} < 0 \\
\end{cases},\right. \\ & \left.- p_i^{(l+1)} \begin{cases}
K'\left[p_{i+1}^{(l+1)}\right], & {\Delta p}_{i}^{i+1} \geq 0\\
0, & {\Delta p}_{i}^{i+1} < 0
\end{cases}\right).
\end{align*}
\endgroup

\begin{lemma}\label{lem:symboldiffusionupstream}
	The matrix sequence $\{h_z^2\, \prescript{{\text{\tiny
				UP}}}{}{L_N}\}_N$ has GLT symbol
	\[ K(p(\psi(x))) (2-2\cos(\theta)),\]
	where $\psi(x)$ is the function mapping $[0,1]$ to the domain of
	definition for the variable $z$.
\end{lemma}

\begin{proof}
	The proof is similar to the one for Lemma~\ref{lem:symboldiffusion}.
	We need to consider that the entries in the rows of the matrices are
	evaluation of the $K(p(z))$ and $p(x)K'(p(z))$ functions on
	different stencils depending on the flux between the interfaces of
	the solution vector $\mathbf{p}^{(k,l+1)}$.
	
	Let us start from the first tridiagonal matrix
	\begin{equation*}
	\begin{split}
	A_N =  & \operatorname{tridiag}\left(- \begin{cases} K\left[p_i^{(l+1)}\right], & {\Delta p}_{i-1}^{i} \geq 0\\
	K\left[p_{i-1}^{(l+1)}\right], & {\Delta p}_{i-1}^{i} < 0
	\end{cases}, \right. \\ & \left. \left( \begin{cases}
	K\left[p_{i+1}^{(l+1)}\right], & {\Delta p}_{i}^{i+1} \geq 0 \\
	K\left[p_i^{(l+1)}\right], & {\Delta p}_{i}^{i+1} < 0
	\end{cases} +  \begin{cases}
	K\left[p_i^{(l+1)}\right], & {\Delta p}_{i-1}^{i} \geq 0 \\
	K\left[p_{i-1}^{(l+1)}\right], & {\Delta p}_{i-1}^{i} < 0
	\end{cases} \right), \right. \\ &\left. -\begin{cases} K\left[p_i^{(l+1)}\right], & {\Delta p}_{i-1}^{i} \geq 0\\
	K\left[p_{i-1}^{(l+1)}\right], & {\Delta p}_{i-1}^{i} < 0
	\end{cases} \right),
	\end{split}
	\end{equation*}
	to produce its GLT symbol, we can consider the matrix sequence
	$D_N(K(p^{(l+1)}))$ $T_N(2-2\cos\theta)$. If we then perform a
	direct comparison of
	$$
	Z_N = A_N -\frac{1}{2}D_N(K(p^{(l+1)})) T_N(2-2\cos\theta),
	$$
	we observe that four different occurrences are possible if we
	consider the previous difference on the $i$th row
	\begin{equation*}
	\left[ Z_N \right]_{i,:} = \begin{cases}
	(0,K[p_{i+1}^{(l+1)}]-K[p_i^{(l+1)}],0),& {\Delta p}_{i-1}^{i} \geq 0  \land {\Delta p}_{i}^{i+1} \geq 0,  \\
	(K[p_i^{(l+1)}]-K[p_{i-1}^{(l+1)}], \\ \qquad K[p_{i+1}^{(l+1)}]+K[p_{i-1}^{(l+1)}]-2K[p_i^{(l+1)}], & {\Delta p}_{i-1}^{i} < 0 \land {\Delta p}_{i}^{i+1} \geq 0, \\ \qquad K[p_i^{(l+1)}])-K[p_{i-1}^{(l+1)}]),  \\
	(0,0,0),& {\Delta p}_{i-1}^{i} \geq 0 \land  {\Delta p}_{i}^{i+1} < 0,  \\
	(K[p_i^{(l+1)}]-K[p_{i-1}^{(l+1)}], \\ \qquad K[p_{i-1}^{(l+1)}]-K[p_i^{(l+1)}], & {\Delta p}_{i-1}^{i} < 0  \land {\Delta p}_{i}^{i+1} < 0, \\ \qquad K[p_i^{(l+1)}]-K[p_{i-1}^{(l+1)}]),  \\
	\end{cases}
	\end{equation*}
	from which we find again that the $1$-norm and the $\infty$-norm of
	the difference
	$$
	A_N -\frac{1}{2}D_N(K(p^{(l+1)})) T_N(2-2\cos\theta)
	$$
	are bounded by a constant multiple of the continuity module
	$\omega_{K(p(z))}(h_z)$ of $K(p(z))$. Thus,
	\[ \| Z_N \| \leq 4 \omega_{K(p(z))}(h_z) \rightarrow 0 \text{ for } N \rightarrow + \infty,
	\].
	Therefore, $Z_N$ is distributed as zero. Then, a direct applications
	of \textbf{GLT3} and \textbf{GLT4} tells us that $A_N \sim_{\rm GLT}
	K(p(\psi(x)))(2-2\cos(\theta))$. In an analogous way, by looking at
	the row-wise differences for the different cases, the same argument
	holds for the other constitutive parts of $\{h_z^2\,
	\prescript{{\text{\tiny UP}}}{}{L_N} \}_N$, and by means of the
	*-algebra properties from \textbf{GLT4}, we can write
	\begin{equation*}
	\begin{split}
	\{ h^2_z \prescript{{\text{\tiny UP}}}{}{L_N}\}_N \sim_{\rm GLT} & K(p(\psi(x)))(2-2\cos(\theta)) + p(\psi(x))K'(p(\psi(x)))(2-2\cos(\theta)) \\
	& - p(\psi(x))K'(p(\psi(x)))(2-2\cos(\theta))\\& = K(p(\psi(x))) (2-2\cos(\theta)). \qedhere
	\end{split}
	\end{equation*}
\end{proof}

\begin{theorem}\label{thm:eigenvaluedistributionwithupstreammean}
	The matrix sequence $\{h_z^2\, \prescript{{\text{\tiny
				UP}}}{}{J_N}\}_N$ is distributed in the eigenvalue sense as the
	function
	\begin{equation*}
	\kappa(x,\theta) = C s'(p(\psi(x))) + K(p(\psi(x))) (2-2\cos(\theta)),
	\end{equation*}
	where $\psi(x)$ is the function mapping $[0,1]$ interval to the domain of definition for the $x$ variable.
\end{theorem}

\begin{proof}
	The proof follows in two step by applying \textbf{GLT4} with the
	symbols obtained in Lemmas~\ref{lem:symboltime},
	\ref{lem:symboldiffusionupstream}, and~\ref{lem:symboltransport}. We
	first find that $\{h_z^2\, \prescript{{\text{\tiny
				ARIT}}}{}{J_N}\}_N\sim_{\rm GLT} \kappa(z,\theta) = C s'(p(\psi(x)))
	+ K(p(\psi(x))) (2-2\cos(\theta))$. Then, by \textbf{GLT2}, that
	holds in virtue of Lemma~\ref{lem:symboltransport}, we have that the
	distribution holds also in the eigenvalue sense.
\end{proof}

\begin{remark}
	The same observation made in
	Remark~\ref{rmk:onlythelaplacianmatters} holds true also for this
	case. Indeed,
	Theorem~\ref{thm:eigenvaluedistributionwithupstreammean} and
	Theorem~\ref{thm:eigenvaluedistributionwitharithmeticmean} predict
	the same spectral distribution for both the discretization choices.
	What we observe is that the spectrum is dominated asymptotically by
	the diffusive character, in a way that regardless how the diffusion
	is discretized.
\end{remark}

\section{The Jacobian matrix for upstream averages}
\label{sec:jacobian-expression}

We give here the complete expression of the Jacobian for the other case of interest, that is the Jacobian represented by the usage of the
upstream-weighted mean in~\eqref{eq:upstream_mean}, where the
interface value is determined at the grid point where the
flux is coming from. We report here the
closed form expression of the Jacobian matrix for that choice. To
shorten the notation, we pose
\begin{equation*}
\begin{split}
{\Delta p}_{i}^{i+1} = p_{i+1,j,k}^{(l)}-p_{i,j,k}^{(l)}, \qquad {\Delta p}_{i-1}^{i} = p_{i,j,k}^{(l)}-p_{i-1,j,k}^{(l)}, \\
{\Delta p}_{j}^{j+1} = p_{i,j+1,k}^{(l)}-p_{i,j,k}^{(l)}, \qquad {\Delta p}_{j-1}^{j} = p_{i,j,k}^{(l)}-p_{i,j-1,k}^{(l)}, \\
{\Delta p}_{k}^{k+1} = p_{i,j,k+1}^{(l)}-p_{i,j,k}^{(l)}, \qquad {\Delta p}_{k-1}^{k} = p_{i,j,k}^{(l)}-p_{i,j,k-1}^{(l)}, \\
\end{split}
\end{equation*}
and we report only the nonzero elements by maintaining the local
$(i,j,k)$-indexes
\begingroup
\allowdisplaybreaks
\begin{align*}
\frac{\partial \boldsymbol{\Phi}}{\partial p_{i,j,k}^{(l)}} = & \frac{\rho  \phi  s'\left(p_{i,j,k}^{(l)}\right)}{\text{$\Delta $t}} \\ & -\frac{p_{i-1,j,k}^{(l)} }{h_x^2}\begin{cases}
K'\left[p_{i,j,k}^{(l)}\right],& {\Delta p}_{i-1}^{i} \geq 0\\
0,& {\Delta p}_{i-1}^{i} < 0
\end{cases} -\frac{p_{i,j-1,k}^{(l)} }{h_y^2}\begin{cases}
K'\left[p_{i,j,k}^{(l)}\right],& {\Delta p}_{j-1}^{j} \geq 0\\
0,& {\Delta p}_{j-1}^{j} < 0
\end{cases} \\ & -\frac{p_{i,j,k-1}^{(l)} }{h_z^2}\begin{cases}
K'\left[p_{i,j,k}^{(l)}\right],& {\Delta p}_{k-1}^{k} \geq 0\\
0,& {\Delta p}_{k-1}^{k} < 0
\end{cases} -\frac{p_{i+1,j,k}^{(l)} }{h_x^2} \begin{cases}
0, & {\Delta p}_{i}^{i+1} \geq 0 \\
K'\left[p_{i,j,k}^{(l)}\right], & {\Delta p}_{i}^{i+1} < 0
\end{cases} \\
& -\frac{p_{i,j+1,k}^{(l)} }{h_y^2} \begin{cases}
0, & {\Delta p}_{j}^{j+1} \geq 0 \\
K'\left[p_{i,j,k}^{(l)}\right], & {\Delta p}_{j}^{j+1} < 0
\end{cases} -\frac{p_{i,j,k+1}^{(l)} }{h_z^2} \begin{cases}
0, & {\Delta p}_{i}^{i+1} \geq 0 \\
K'\left[p_{i,j,k}^{(l)}\right], & {\Delta p}_{k}^{k+1} < 0
\end{cases} \\
& +p_{i,j,k}^{(l)} \left(\frac{1}{h_x^2}\left( \begin{cases}
0,&{\Delta p}_{i}^{i+1} \geq 0,\\
K'\left[p_{i,j,k}^{(l)}\right],& {\Delta p}_{i}^{i+1} < 0,
\end{cases}
+
\begin{cases}
K'\left[p_{i,j,k}^{(l)}\right],& {\Delta p}_{i-1}^{i} \geq 0,\\
0,& {\Delta p}_{i-1}^{i} < 0,
\end{cases}
\right) \right. \\ &
+ \frac{1}{h_y^2}\left( \begin{cases}
0,&{\Delta p}_{j}^{j+1} \geq 0,\\
K'\left[p_{i,j,k}^{(l)}\right],& {\Delta p}_{j}^{j+1} < 0,
\end{cases}
+
\begin{cases}
K'\left[p_{i,j,k}^{(l)}\right], & {\Delta p}_{j-1}^{j} \geq 0,\\
0, & {\Delta p}_{j-1}^{j} < 0,
\end{cases}
\right) \\ &\left.
+\frac{1}{h_z^2}\left( \begin{cases}
0,&{\Delta p}_{k}^{k+1} \geq 0,\\
K'\left[p_{i,j,k}^{(l)}\right],& {\Delta p}_{k}^{k+1} < 0,
\end{cases}
+
\begin{cases}
K'\left[p_{i,j,k}^{(l)}\right],& {\Delta p}_{k-1}^{k} \geq 0,\\
0,& {\Delta p}_{k-1}^{k} < 0,
\end{cases}
\right) \right) \\
& + \frac{1}{h_x^2}\left(\begin{cases}
K\left[p_{i+1,j,k}^{(l)}\right], & {\Delta p}_{i}^{i+1} \geq 0,\\
K\left[p_{i,j,k}^{(l)}\right], & {\Delta p}_{i}^{i+1} < 0,
\end{cases}
+
\begin{cases}
K\left[p_{i,j,k}^{(l)}\right], & {\Delta p}_{i-1}^{i} \geq 0\\
K\left[p_{i-1,j,k}^{(l)}\right], & {\Delta p}_{i-1}^{i} < 0,
\end{cases}\right) \\
& + \frac{1}{h_y^2}\left(\begin{cases}
K\left[p_{i,j+1,k}^{(l)}\right], & {\Delta p}_{j}^{j+1} \geq 0,\\
K\left[p_{i,j,k}^{(l)}\right], & {\Delta p}_{j}^{j+1} < 0,
\end{cases}
+
\begin{cases}
K\left[p_{i,j,k}^{(l)}\right], & {\Delta p}_{j-1}^{j} \geq 0\\
K\left[p_{i,j-1,k}^{(l)}\right], & {\Delta p}_{j-1}^{j} < 0,
\end{cases}\right) \stepcounter{equation}\tag{\theequation}\label{eq:jacobian_upstreamaverage} \\
& + \frac{1}{h_z^2}\left(\begin{cases}
K\left[p_{i,j,k+1}^{(l)}\right], & {\Delta p}_{k}^{k+1} \geq 0,\\
K\left[p_{i,j,k}^{(l)}\right], & {\Delta p}_{k}^{k+1} < 0,
\end{cases}
+
\begin{cases}
K\left[p_{i,j,k}^{(l)}\right], & {\Delta p}_{k-1}^{k} \geq 0\\
K\left[p_{i,j,k-1}^{(l)}\right], & {\Delta p}_{k-1}^{k} < 0,
\end{cases}\right), \\
\frac{\partial \boldsymbol{\Phi}}{\partial p_{i-1,j,k}^{(l)}} = & \frac{p_{i,j,k}^{(l)}}{h_x^2} \begin{cases}
0, & {\Delta p}_{i-1}^{i} \geq 0 \\
K'\left[p_{i-1,j,k}^{(l)}\right], & {\Delta p}_{i-1}^{i} < 0
\end{cases} - \frac{p_{i-1,j,k}^{(l)}}{h_x^2} \begin{cases}
0, & {\Delta p}_{i-1}^{i} \geq 0 \\
K'\left[p_{i-1,j,k}^{(l)}\right], & {\Delta p}_{i-1}^{i} \geq 0
\end{cases} \\
& - \frac{1}{h_x^2} \begin{cases}
K\left[p_{i,j,k}^{(l)}\right], & {\Delta p}_{i-1}^{i} \geq 0\\
K\left[p_{i-1,j,k}^{(l)}\right], & {\Delta p}_{i-1}^{i} < 0,
\end{cases},\\
\frac{\partial \boldsymbol{\Phi}}{\partial p_{i+1,j,k}^{(l)}} = & \frac{p_{i,j,k}^{(l)}}{h_x^2} \begin{cases}
K'\left[p_{i+1,j,k}^{(l)}\right], & {\Delta p}_{i}^{i+1} \geq 0 \\
0, & {\Delta p}_{i}^{i+1} < 0\\
\end{cases} -\frac{p_{i+1,j,k}^{(l)}}{h_x^2}\begin{cases}
K'\left[p_{i+1,j,k}^{(l)}\right],& {\Delta p}_{i}^{i+1} \geq 0\\
0,& {\Delta p}_{i}^{i+1} < 0
\end{cases} \\
&-\frac{1}{h_x^2}\begin{cases}
K\left[p_{i+1,j,k}^{(l)}\right], & {\Delta p}_{i}^{i+1} \geq 0,\\
K\left[p_{i,j,k}^{(l)}\right], & {\Delta p}_{i}^{i+1} < 0,
\end{cases}\\
\frac{\partial \boldsymbol{\Phi}}{\partial p_{i,j-1,k}^{(l)}} = & \frac{p_{i,j,k}^{(l)}}{h_y^2} \begin{cases}
0, & {\Delta p}_{j-1}^{j} \geq 0 \\
K'\left[p_{i,j-1,k}^{(l)}\right], & {\Delta p}_{j-1}^{j} < 0
\end{cases} - \frac{p_{i-1,j,k}^{(l)}}{h_y^2} \begin{cases}
0, & {\Delta p}_{j-1}^{j} \geq 0 \\
K'\left[p_{i,j-1,k}^{(l)}\right], & {\Delta p}_{j-1}^{j} \geq 0
\end{cases} \\
& - \frac{1}{h_y^2} \begin{cases}
K\left[p_{i,j,k}^{(l)}\right], & {\Delta p}_{j-1}^{j} \geq 0\\
K\left[p_{i,j-1,k}^{(l)}\right], & {\Delta p}_{j-1}^{j} < 0,
\end{cases},\\
\frac{\partial \boldsymbol{\Phi}}{\partial p_{i,j+1,k}^{(l)}} = & \frac{p_{i,j,k}^{(l)}}{h_y^2} \begin{cases}
K'\left[p_{i,j+1,k}^{(l)}\right], & {\Delta p}_{j}^{j+1} \geq 0 \\
0, & {\Delta p}_{j}^{j+1} < 0\\
\end{cases} -\frac{p_{i,j+1,k}^{(l)}}{h_y^2}\begin{cases}
K'\left[p_{i,j+1,k}^{(l)}\right],& {\Delta p}_{j}^{j+1} \geq 0\\
0,& {\Delta p}_{j}^{j+1} < 0
\end{cases} \\
&-\frac{1}{h_y^2}\begin{cases}
K\left[p_{i,j+1,k}^{(l)}\right], & {\Delta p}_{j}^{j+1} \geq 0,\\
K\left[p_{i,j,k}^{(l)}\right], & {\Delta p}_{j}^{j+1} < 0,
\end{cases}\\
\frac{\partial \boldsymbol{\Phi}}{\partial p_{i,j,k-1}^{(l)}} = & \frac{p_{i,j,k}^{(l)}}{h_z^2} \begin{cases}
0, & {\Delta p}_{k-1}^{k} \geq 0 \\
K'\left[p_{i,j,k-1}^{(l)}\right], & {\Delta p}_{k-1}^{k} < 0
\end{cases} - \frac{p_{i,j,k-1}^{(l)}}{h_y^2} \begin{cases}
0, & {\Delta p}_{k-1}^{k} \geq 0 \\
K'\left[p_{i,j,k-1}^{(l)}\right], & {\Delta p}_{k-1}^{k} \geq 0
\end{cases} \\
& - \frac{1}{h_z^2} \begin{cases}
K\left[p_{i,j,k}^{(l)}\right], & {\Delta p}_{k-1}^{k} \geq 0\\
K\left[p_{i,j,k-1}^{(l)}\right], & {\Delta p}_{k-1}^{j} < 0,
\end{cases} + \frac{K'\left[p_{i,j,k-1}^{(l)}\right]}{2 h_z},\\
\frac{\partial \boldsymbol{\Phi}}{\partial p_{i,j,k+1}^{(l)}} = & \frac{p_{i,j,k}^{(l)}}{h_z^2} \begin{cases}
K'\left[p_{i,j,k+1}^{(l)}\right], & {\Delta p}_{k}^{k+1} \geq 0 \\
0, & {\Delta p}_{k}^{k+1} < 0\\
\end{cases} -\frac{p_{i,j,k+1}^{(l)}}{h_z^2}\begin{cases}
K'\left[p_{i,j,k+1}^{(l)}\right],& {\Delta p}_{k}^{k+1} \geq 0\\
0,& {\Delta p}_{k}^{k+1} < 0
\end{cases} \\
&-\frac{1}{h_z^2}\begin{cases}
K\left[p_{i,j,k+1}^{(l)}\right], & {\Delta p}_{k}^{k+1} \geq 0,\\
K\left[p_{i,j,k}^{(l)}\right], & {\Delta p}_{k}^{k+1} < 0,
\end{cases} - \frac{K'\left[p_{i,j,k+1}^{(l)}\right]}{2 h_z}.
\end{align*}
\endgroup

\section{A comparison between spectral characterization and true eigenvalues}
\label{theoreticalresults}

To illustrate the results depicted in Theorem~\ref{thm:eigenvaluedistributionwitharithmeticmean}. We
consider a small 1D example for having a small MATLAB implementation, and being able of looking the comparison between the eigenvalues of the Jacobian matrices for the full Newton method and their distribution function. That is, we check how good the asymptotic description of the spectrum behaves with respect to the true eigenvalues.

Specifically, we consider a 1D water infiltration test\cite{doi:10.2136/sssaj1977.03615995004100020024x} for a column of soil with parameters for the Van Genuchten model~\eqref{eq:s(p)}-\eqref{eq:K(p)} selected as
\begin{equation}\label{eq:coefficientspectrum1}
\begin{split}
\alpha = \texttt{1.611e+6},\, \beta  = 3.96,\, \gamma = 4.74,\, a     = \texttt{1.175e+6},\\ S_s    = 0.287,\, S_r    = 0.075,\, K_s    = 0.00944 \si{cm/s},
\end{split}
\end{equation}
with initial condition $p(z,0) = -61.5 \si{cm}$, boundary conditions $p(40 \si{cm},t) = -20.7 \si{cm}$ and $p(0,t) = -61.5 \si{cm}$, i.e., we are assuming the vertical dimension to be positive in the upward direction. In Figure~\ref{fig:eigenvaluesexample1} we have reported the spectrum of the Jacobian for different time steps and different iterates of the full Newton method with no line search.
\begin{figure}[htbp]
	\centering
	\includegraphics[width=\columnwidth]{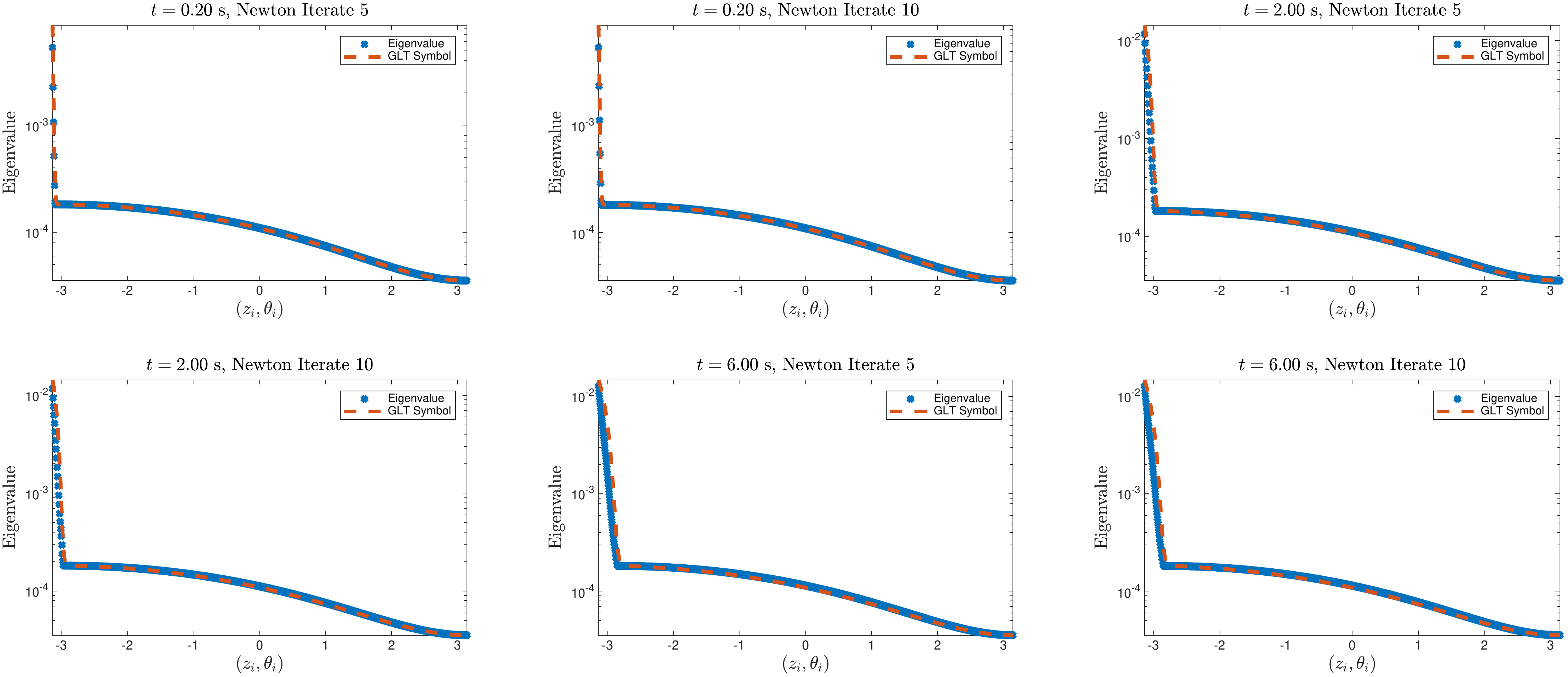}
	\caption{Arithmetic average. Comparison of the eigenvalues and spectral symbol for the coefficients in~\eqref{eq:coefficientspectrum1} with $h_z = 40/(N-1)$, $\Delta t = 0.1$, and $N=800$ on different time steps and for different iterates of the full Newton method.\label{fig:eigenvaluesexample1}}
\end{figure}
As we observe, even for a small grid size thus namely far away from the asymptotic regime, there is a good matching between the predicted eigenvalues and the computed ones.

We consider the same test problem also for the case in which we use the upstream-weighted mean of the values of $K$ in~\eqref{eq:K(p)}. Again the result proved in Theorem~\ref{thm:eigenvaluedistributionwithupstreammean} is clearly depicted in Figure~\ref{fig:eigenvaluesexample2}, in which we observe a good accordance between the spectral symbol and the computed eigenvalues already at coarse grid.
\begin{figure}[htbp]
	\centering
	\includegraphics[width=\columnwidth]{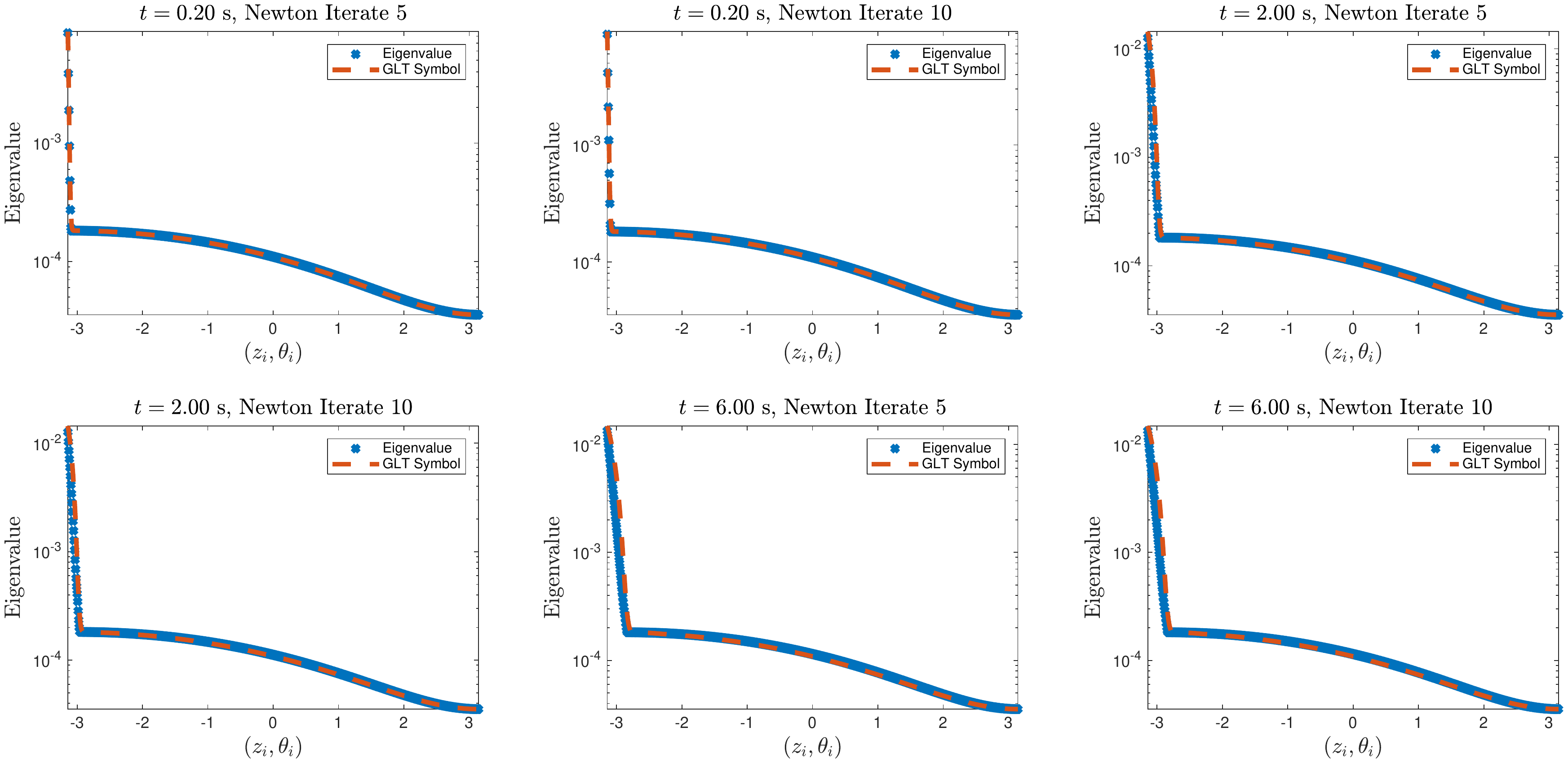}
	\caption{Upstream-weighted average. Comparison of the eigenvalues and spectral symbol for the coefficients in~\eqref{eq:coefficientspectrum1} with $h_z = 40/(N-1)$, $\Delta t = 0.1$, and $N=800$ on different time steps and for different iterates of the full Newton method.\label{fig:eigenvaluesexample2}}
\end{figure}

To have a comparison in between the approximate and full sequence we consider also the only preconditioner we can apply on both
sequence, namely the AS preconditioner, and compare the average number of iteration across the time steps in the \emph{weak scaling} regime 
described in Section~\ref{sec:weak_scalability}. From the results in Table~\ref{tab:spectral_equiv_iteration} we observe readily that the 
spectrally equivalent sequence clearly catches the properties we are interested in.

\begin{table}[htbp]
    \centering
    \begin{tabular}{ccc}
        \toprule
        $n_p$ & Original sequence & Approximating sequence \\
        \midrule
           1&         212&         212\\
           4&         267&         267\\
          16&         277&         285\\
          64&         296&         296\\
         256&         217&         217\\
        1024&         296&         296\\
        \bottomrule
    \end{tabular}
    \caption{Weak scaling. Comparison of the (rounded) average number of linear iteration employed by the AS preconditioner when built directly
    on the original sequence and when built on the approximating sequence obtained from the spectral analysis.}
    \label{tab:spectral_equiv_iteration}
\end{table}

\section{Weak scalability} 
\label{sec:algscalability}

In this section we complement the results discussed in Section~\ref{sec:weak_scalability} by considering further details on the weak-scalability analysis. 

In first place, we want to investigate numerically the consequences of Theorem~\ref{thm:spectraldescription}. Indeed, from the theoretical analysis we expect that if we would use just the discretization of the diffusive part of the equation, then a cluster of eigenvalues at $1$ would appear and \emph{suggest} good performances for the underlying Krylov solver. Since this is not feasible, we resorted to using either the multigrid hierarchies to further approximate the diffusive part. Therefore, we investigate here numerically the capability of such approximating sequence with respect to the original one. To have an homogeneous situation with respect to the coarse grid, and diminish the number of variables that could change the interpretation of the result, we resort to using just 30 iterations of a Block-Jacobi method with the ILU(0) factorization as a coarse solver.
\begin{table}[htb]
	\centering
	\subfloat[VDSVMB]{
		\begin{tabular}{rcccc}
			\toprule
			&  Avg. Nl. &          & \multicolumn{2}{c}{Avg. L. It.s} \\ 
			$n_p$ &  It.s     & N Jac.s & & R \\
			\midrule
			1 & 3 & 12 & 58 & 57 \\
			4 & 3 & 12 & 69 & 69 \\
			16 & 3 & 12 & 68 & 68 \\
			64 & 3 & 12 & 57 & 59 \\
			256 & 3 & 12 & 56 & 58 \\
			1024 & 3 & 12 & 57 & 57 \\
			4096 & 3 & 12 & 56 & 56 \\
			\bottomrule
		\end{tabular}
	}\hfil
	\subfloat[VSMATCH]{
		\begin{tabular}{rcccc}
			\toprule
			&  Avg. Nl. &          & \multicolumn{2}{c}{Avg. L. It.s} \\ 
			$n_p$ &  It.s     & N Jac.s & & R \\
			\midrule
			1 & 3 & 11 & 82 & 83 \\
			4 & 3 & 12 & 87 & 73 \\
			16 & 3 & 12 & 83 & 75 \\
			64 & 3 & 12 & 79 & 87 \\
			256 & 3 & 12 & 77 & 82 \\
			1024 & 3 & 12 & 75 & 75 \\
			4096 & 3 & 12 & 73 & 77 \\
			\bottomrule
		\end{tabular}
	}
	\caption{Algorithmic scalability for the Richards test problem. The average number of linear iteration per Newton iteration is computed taking the average of the ratio of linear iteration per Newton step across the time steps. }\label{tab:richards}
\end{table}
From the results collected in Table~\ref{tab:richards}, we observe that both the multigrid strategies do manage in keeping the number of linear iteration more stable and smaller than the one-level AS preconditioner. This makes for an empirical confirmation of the two facts we want to ascertain. On one hand the spectral information of the matrix sequence  $\{J_{\mathbf{N}}\}_{\mathbf{N}}$ are confirmed to be captured by the symmetric approximation, secondly, the multigrid hierarchies give a reliable approximation of the effect of inverting the theoretically implied sequence.

\subsection{Finer details on the build and solve phases}
\label{sec:details-on-build-and-solve}

To deepen the analysis, we look with a finer degree of detail at the performances of every slices of the bar-plots in Figure~\ref{fig:timespentpercentage}. Let us start from the bottom and look at the build phase of the Jacobian matrices; Figure~\ref{fig:matrix-build-scaling}. Up to $4096$ cores the assembly phase stays around the 70\% of efficiency, and is only sligthly diminished to around the 60\% on $8192$ cores.
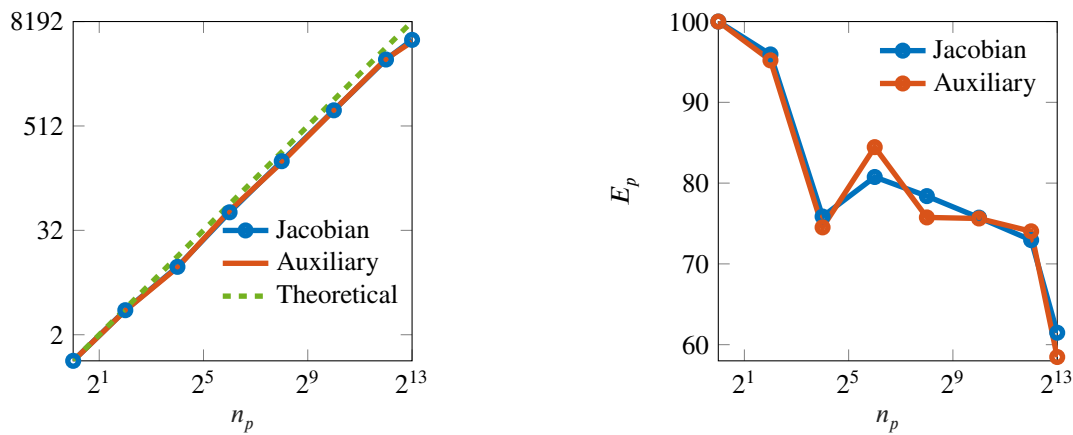
\begin{figure}[htbp]
	\centering
%
%
\definecolor{mycolor1}{rgb}{0.00000,0.44700,0.74100}%
\definecolor{mycolor2}{rgb}{0.85000,0.32500,0.09800}%
\definecolor{mycolor4}{rgb}{0.45000,0.69400,0.12500}%
\begin{tikzpicture}

\begin{axis}[%
width=0.25\columnwidth,
height=0.25\columnwidth,
at={(0.758in,0.481in)},
scale only axis,
xmode=log,
log basis x=2,
xmin=1,
xmax=8192,
xminorticks=true,
xlabel style={font=\color{white!15!black}},
xlabel={$n_p$},
ymode=log,
log basis y=2,
ymin=1,
ymax=8192,
yminorticks=true,
yticklabels={2,32,512,8192},
axis background/.style={fill=white},
title style={font=\bfseries},
legend style={legend cell align=left, align=left, draw=none, fill=none,at={(1,0.45)}}
]
\addplot [color=mycolor1, line width=2.0pt, mark=o, mark options={solid, mycolor1}]
  table[row sep=crcr]{%
1	1\\
4	3.83723284819041\\
16	12.137956417059\\
64	51.6849053069683\\
256	200.682418191902\\
1024	775.454764749627\\
4096	2987.87848797592\\
8192	5036.3814534813\\
};
\addlegendentry{Jacobian}

\addplot [color=mycolor2, line width=2.0pt]
  table[row sep=crcr]{%
1	1\\
4	3.80787968060231\\
16	11.9231982514745\\
64	54.0505771140735\\
256	193.926226162632\\
1024	774.358209060986\\
4096	3032.39192255896\\
8192	4789.98932807705\\
};
\addlegendentry{Auxiliary}

\addplot [color=mycolor4, line width=2.0pt, dashed]
  table[row sep=crcr]{%
1	1\\
4	4\\
16	16\\
64	64\\
256	256\\
1024 1024\\
4096 4096\\
8192 8192\\
};
\addlegendentry{Theoretical}

\end{axis}

\begin{axis}[%
width=0.25\columnwidth,
height=0.25\columnwidth,
at={(4.1in,0.481in)},
scale only axis,
xmode=log,
log basis x=2,
xmin=1,
xmax=8192,
xminorticks=true,
xlabel style={font=\color{white!15!black}},
xlabel={$n_p$},
ylabel={$E_p$},
ymin=58,
ymax=100,
axis background/.style={fill=white},
title style={font=\bfseries},
legend style={legend cell align=left, align=left, draw=none, fill=none}
]
\addplot [color=mycolor1, line width=2.0pt, mark=o, mark options={solid, mycolor1}]
  table[row sep=crcr]{%
1	100\\
4	95.9308212047603\\
16	75.8622276066187\\
64	80.757664542138\\
256	78.3915696062119\\
1024	75.7280043700808\\
4096	72.9462521478497\\
8192	61.4792657895666\\
};
\addlegendentry{Jacobian}

\addplot [color=mycolor2, line width=2.0pt, mark=o, mark options={solid, mycolor2}]
  table[row sep=crcr]{%
1	100\\
4	95.1969920150578\\
16	74.5199890717156\\
64	84.4540267407398\\
256	75.752432094778\\
1024	75.620918853612\\
4096	74.0330059218496\\
8192	58.471549415003\\
};
\addlegendentry{Auxiliary}

\end{axis}
\end{tikzpicture}%
	\caption{Weak scaling. Speedup (left) and scaled efficiency (right) for setup of the Jacobian matrices.}
	\label{fig:matrix-build-scaling}
\end{figure}
Indeed, even if we are using a method that needs to assembly the matrices for building the preconditioners, we observe that this doesn't hampers the overall efficiency of the method. This is even more true if we also consider the efficiency of the function evaluation routine (that is instead implemented in a matrix-free way) given in Figure~\ref{fig:feval-scaling}. We notice that this mimics the overall efficiency of the whole procedure.
\begin{figure}[htbp]
	\centering
%
%
\definecolor{mycolor1}{rgb}{0.00000,0.44700,0.74100}%
\definecolor{mycolor4}{rgb}{0.45000,0.69400,0.12500}%
\begin{tikzpicture}

\begin{axis}[%
width=0.25\columnwidth,
height=0.25\columnwidth,
at={(0.758in,0.481in)},
scale only axis,
xmode=log,
log basis x = 2,
log basis y = 2,
xmin=1,
xmax=8192,
xminorticks=true,
xlabel style={font=\color{white!15!black}},
xlabel={$n_p$},
ymode=log,
ymin=1,
ymax=8192,
yminorticks=true,
axis background/.style={fill=white},
title style={font=\bfseries},
yticklabels={2,32,512,8192},
legend style={legend cell align=left, align=left, draw=none, fill=none, at={(0.9,1.0)}}
]
\addplot [color=mycolor1, line width=2.0pt]
  table[row sep=crcr]{%
1	1\\
4	3.92528752534871\\
16	13.2923562409231\\
64	52.0257254153154\\
256	190.642312898308\\
1024	682.63402786094\\
4096	2417.38825896879\\
8192	4779.41204355111\\
};
\addlegendentry{Func. Evaluation}

\addplot [color=mycolor4, line width=2.0pt, dashed]
  table[row sep=crcr]{%
1	1\\
4	4\\
16	16\\
64	64\\
256	256\\
1024 1024\\
4096 4096\\
8192 8192\\
};
\addlegendentry{Theoretical}

\end{axis}

\begin{axis}[%
width=0.25\columnwidth,
height=0.25\columnwidth,
at={(3.327in,0.481in)},
scale only axis,
xmode=log,
log basis x = 2,
xmin=1,
xmax=8192,
xminorticks=true,
xlabel style={font=\color{white!15!black}},
xlabel={$n_p$},
ymin=55,
ymax=100,
axis background/.style={fill=white},
title style={font=\bfseries},
ylabel={$E_p$},
legend style={legend cell align=left, align=left, draw=white!15!black}
]
\addplot [color=mycolor1, line width=2.0pt]
  table[row sep=crcr]{%
1	100\\
4	98.1321881337178\\
16	83.0772265057695\\
64	81.2901959614304\\
256	74.4696534759016\\
1024	66.663479283295\\
4096	59.0182680412302\\
8192	58.3424321722548\\
};

\end{axis}
\end{tikzpicture}%
	\caption{Weak scaling. Speedup (left) and scaled efficiency (right) for the evaluation of the $\boldsymbol{\Phi}$ function.}
	\label{fig:feval-scaling}
\end{figure}
As it usually happens in nonlinear processes, what characterizes the overall performance is indeed the number and efficiency of the function evaluation calls. That, as we see from Table~\ref{tab:weak-scaling-iterations}, are very stable with respect to the number of employed cores when the VSMATCH preconditioner is used for Newton correction, while shows some oscillations for the other preconditioners.
\begin{table}[htbp]
	\centering
	\begin{tabular}{||c||c|c||c|c||c|c||}
		\toprule
		& \multicolumn{2}{c||}{VDSVBM} & \multicolumn{2}{c||}{VSMATCH} & \multicolumn{2}{c||}{AS} \\
		\midrule
		$n_p$   & N Jac.s & NLin It.s & N Jac.s & NLin It.s & N Jac.s & NLin It.s  \\
		\midrule
		1 & 3 & 37 & 3 & 36 & 3 & 40 \\
		4 & 3 & 38 & 3 & 38 & 3 & 36 \\
		16 & 3 & 38 & 3 & 38 & 3 & 40 \\
		64 & 3 & 37 & 3 & 38 & 4 & 37 \\
		256 & 3 & 37 & 3 & 38 & 4 & 39 \\
		1024 & 3 & 39 & 3 & 38 & 4 & 41 \\
		4096 & 3 & 41 & 3 & 38 & 4 & 47 \\
		8192 & 3 & 40 & 3 & 38 & 4 & 48 \\
		\bottomrule
	\end{tabular}
	\caption{Weak scaling. Number of nonlinear iterations (NLin It.s), and number of computed Jacobians (N Jac.s) for the three preconditioners.}
	\label{tab:weak-scaling-iterations}
\end{table}

Finally, let us look further into the scalability performances of the solution of the linear systems associated with the Newton correction. In our strategy this is divided into three separate phases. A \emph{complete setup} in which we build the hierarchies for the first time, an \emph{update} phase in which we recompute just the smoothers when there has been a change in the Jacobian, and the \emph{application phase} inside the Krylov iterative method. 
For what concern the first two phases, Figures~\ref{fig:vsbm-setup}, \ref{fig:VSMATCH3-setup} show that the update strategy reduces effectively the time needed to prepare the preconditioner, while keeping a reasonable quality for the aggregation procedure.
\begin{figure}[htbp]
	\centering
	\subfloat[VDSVMB Setup calls.\label{fig:vsbm-setup}]{
%
%
\definecolor{mycolor1}{rgb}{0.00000,0.44700,0.74100}%
\definecolor{mycolor2}{rgb}{0.85000,0.32500,0.09800}%
\definecolor{mycolor3}{rgb}{0.92900,0.69400,0.12500}%
\definecolor{mycolor4}{rgb}{0.49400,0.18400,0.55600}%
\definecolor{mycolor5}{rgb}{0.46600,0.67400,0.18800}%
\definecolor{mycolor6}{rgb}{0.30100,0.74500,0.93300}%
\definecolor{mycolor7}{rgb}{0.63500,0.07800,0.18400}%
\definecolor{mycolor7}{rgb}{0.63500,0.07800,0.18400}%
\definecolor{mycolor8}{rgb}{0.92900,0.07800,0.18400}%
\begin{tikzpicture}

\begin{axis}[%
width=0.25\columnwidth,
height=0.25\columnwidth,
at={(0.758in,0.481in)},
scale only axis,
xmin=1,
xmax=3,
xlabel style={font=\color{white!15!black}},
ymin=0,
ymax=1.5,
yminorticks=false,
xtick={1, 2, 3},
ylabel style={font=\color{white!15!black}},
axis background/.style={fill=white},
title style={font=\bfseries},
]
\addplot [color=mycolor1, line width=2.0pt, mark=o, mark options={solid, mycolor1}]
  table[row sep=crcr]{%
1	0.2382024\\
2	0.07840671\\
3	0.0783177\\
};

\addplot [color=mycolor2, line width=2.0pt, mark=o, mark options={solid, mycolor2}]
  table[row sep=crcr]{%
1	0.2807704\\
2	0.0904638\\
3	0.09080171\\
};

\addplot [color=mycolor3, line width=2.0pt, mark=o, mark options={solid, mycolor3}]
  table[row sep=crcr]{%
1	0.346576\\
2	0.1084237\\
3	0.1091664\\
};

\addplot [color=mycolor4, line width=2.0pt, mark=o, mark options={solid, mycolor4}]
  table[row sep=crcr]{%
1	0.4072839\\
2	0.109827\\
3	0.1088391\\
};

\addplot [color=mycolor5, line width=2.0pt, mark=o, mark options={solid, mycolor5}]
  table[row sep=crcr]{%
1	0.4125552\\
2	0.1394439\\
3	0.1253072\\
};

\addplot [color=mycolor6, line width=2.0pt, mark=o, mark options={solid, mycolor6}]
  table[row sep=crcr]{%
1	0.5308499\\
2	0.1554613\\
3	0.1329764\\
};

\addplot [color=mycolor7, line width=2.0pt, mark=o, mark options={solid, mycolor7}]
  table[row sep=crcr]{%
1	0.5470355\\
2	0.1608376\\
3	0.168339\\
};

\addplot [color=mycolor8, line width=2.0pt, mark=o, mark options={solid, mycolor8}]
  table[row sep=crcr]{%
1	0.7213076\\
2	0.1847844\\
3	0.1833497\\
};

\end{axis}
\end{tikzpicture}%
}
	\subfloat[VSMATCH Setup calls.\label{fig:VSMATCH3-setup}]{
%
%
\definecolor{mycolor1}{rgb}{0.00000,0.44700,0.74100}%
\definecolor{mycolor2}{rgb}{0.85000,0.32500,0.09800}%
\definecolor{mycolor3}{rgb}{0.92900,0.69400,0.12500}%
\definecolor{mycolor4}{rgb}{0.49400,0.18400,0.55600}%
\definecolor{mycolor5}{rgb}{0.46600,0.67400,0.18800}%
\definecolor{mycolor6}{rgb}{0.30100,0.74500,0.93300}%
\definecolor{mycolor7}{rgb}{0.63500,0.07800,0.18400}%
\definecolor{mycolor8}{rgb}{0.92900,0.07800,0.18400}%
\begin{tikzpicture}

\begin{axis}[%
width=0.25\columnwidth,
height=0.25\columnwidth,
at={(0.758in,0.481in)},
scale only axis,
unbounded coords=jump,
xmin=1,
xmax=3,
xlabel style={font=\color{white!15!black}},
ymin=0,
ymax=1.52,
yminorticks=false,
ylabel style={font=\color{white!15!black}},
xtick={1, 2, 3},
axis background/.style={fill=white},
title style={font=\bfseries},
legend columns=1,
legend style={at={(1.7,1)}, legend cell align=left, align=left, draw=none, fill=none}
]
\addplot [color=mycolor1, line width=2.0pt, mark=o, mark options={solid, mycolor1}]
  table[row sep=crcr]{%
1	0.4423297\\
2	0.1027059\\
3	nan\\
};
\addlegendentry{$n_p$ = 1}

\addplot [color=mycolor2, line width=2.0pt, mark=o, mark options={solid, mycolor2}]
  table[row sep=crcr]{%
1	0.5715803\\
2	0.1135213\\
3	0.1147068\\
};
\addlegendentry{$n_p$ = 4}

\addplot [color=mycolor3, line width=2.0pt, mark=o, mark options={solid, mycolor3}]
  table[row sep=crcr]{%
1	0.7753036\\
2	0.1410884\\
3	0.1410647\\
};
\addlegendentry{$n_p$ = 16}

\addplot [color=mycolor4, line width=2.0pt, mark=o, mark options={solid, mycolor4}]
  table[row sep=crcr]{%
1	0.8858063\\
2	0.1497689\\
3	0.1638553\\
};
\addlegendentry{$n_p$ = 64}

\addplot [color=mycolor5, line width=2.0pt, mark=o, mark options={solid, mycolor5}]
  table[row sep=crcr]{%
1	1.093476\\
2	0.1644287\\
3	0.1635126\\
};
\addlegendentry{$n_p$ = 256}

\addplot [color=mycolor6, line width=2.0pt, mark=o, mark options={solid, mycolor6}]
  table[row sep=crcr]{%
1	1.012273\\
2	0.1705394\\
3	0.1553091\\
};
\addlegendentry{$n_p$ = 1024}

\addplot [color=mycolor7, line width=2.0pt, mark=o, mark options={solid, mycolor7}]
  table[row sep=crcr]{%
1	1.283158\\
2	0.2196235\\
3	0.1889707\\
};
\addlegendentry{$n_p$ = 4096}

\addplot [color=mycolor8, line width=2.0pt, mark=o, mark options={solid, mycolor8}]
  table[row sep=crcr]{%
1	1.519814\\
2	0.2766051\\
3	0.2606578\\
};
\addlegendentry{$n_p$ = 8192}

\end{axis}
\end{tikzpicture}%
}

	\caption{Weak scaling. Setup phases of the preconditioners in which subsequent calls are updates.}
	\label{fig:setup_phases}
\end{figure}
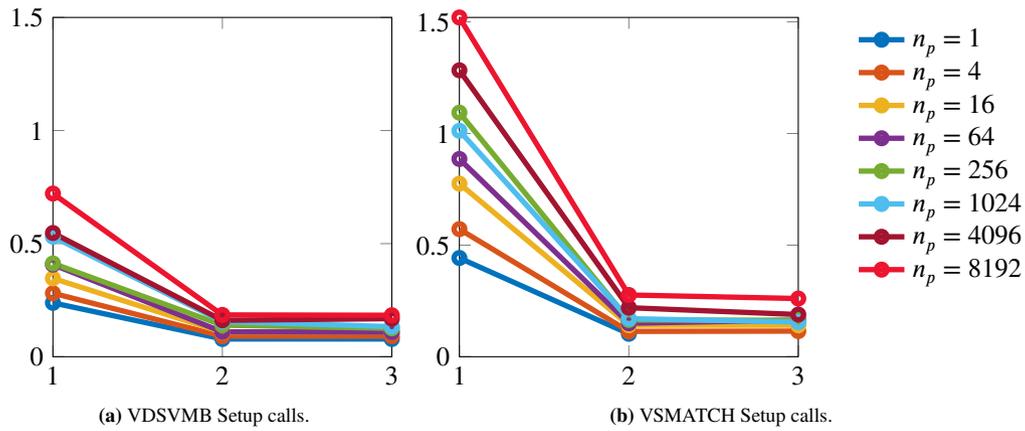
In order to better understand the results on the execution time per iteration of the AMG preconditioners discussed in the paper, we also report (average) \emph{operator complexity} of the two different methods. We recall that the operator complexity of an AMG method is defined as
\[\text{opc}=\frac{\sum_{l=0}^{\text{nl}-1}\operatorname{nnz}(A_{l})}{\operatorname{nnz}(A_0)} > 1,\]
and is usually both a measure of memory occupancy and of the time complexity for the method application in a V-cycle\cite{dambra2020b}.
\begin{figure}[htbp]
	\centering
	\subfloat[Average operator complexity \label{fig:opcomplexity}]{
%
%
\definecolor{mycolor1}{rgb}{0.00000,0.44700,0.74100}%
\definecolor{mycolor2}{rgb}{0.85000,0.32500,0.09800}%
\begin{tikzpicture}

\begin{axis}[%
width=0.25\columnwidth,
height=0.25\columnwidth,
at={(0.758in,0.481in)},
scale only axis,
xmode=log,
log basis x=2,
xmin=1,
xmax=8192,
xminorticks=true,
xlabel style={font=\color{white!15!black}},
xlabel={$n_p$},
ymin=1,
ymax=2.5,
yminorticks=true,
axis background/.style={fill=white},
legend style={legend cell align=left, align=left, draw=none,fill=none}
]
\addplot [color=mycolor1, line width=2.0pt, mark=o, mark options={solid, mycolor1}]
  table[row sep=crcr]{%
1	1.37484133915575\\
4	2.10096816378066\\
16	2.11501664254703\\
64	2.10763544668588\\
256	2.1120554299406\\
1024	2.1120554299406\\
4096	2.11392396798534\\
8192	2.11484076884051\\
};
\addlegendentry{VSMATCH}

\addplot [color=mycolor2, line width=2.0pt, mark=x, mark options={solid, mycolor2}]
  table[row sep=crcr]{%
1	1.56608879184862\\
4	1.60895337301587\\
16	1.59689435600579\\
64	1.61558064661383\\
256	1.61940458513319\\
1024	1.62137008872346\\
4096	1.62204706718553\\
8192	1.62137167028067\\
};
\addlegendentry{VDSVMB}

\end{axis}
\end{tikzpicture}%
}
	\subfloat[Average time $T (s)$ for linear solver (left) and scaled efficiency (right)\label{fig:linearsystem-scaling}]{
%
%
\definecolor{mycolor1}{rgb}{0.00000,0.44700,0.74100}%
\definecolor{mycolor2}{rgb}{0.85000,0.32500,0.09800}%
\definecolor{mycolor3}{rgb}{0.92900,0.69400,0.12500}%
\definecolor{mycolor4}{rgb}{0.49400,0.18400,0.55600}%
\begin{tikzpicture}

\begin{axis}[%
width=0.25\columnwidth,
height=0.25\columnwidth,
at={(0.758in,0.481in)},
scale only axis,
xmode=log,
xmin=1,
xmax=8192,
log basis x=2,
xminorticks=true,
xlabel style={font=\color{white!15!black}},
xlabel={$n_p$},
ymode=log,
ymin=1e-1,
ymax=2,
yminorticks=true,
axis background/.style={fill=white},
title style={font=\bfseries},
legend style={at={(1,0.5)}, legend cell align=left, align=left, draw=none, fill=none}
]
\addplot [color=mycolor1, line width=2.0pt, mark=o, mark options={solid, mycolor1}]
  table[row sep=crcr]{%
1	0.526637444722222\\
4	0.635683591315789\\
16	0.773930532105263\\
64	0.805539252631579\\
256	0.890682915789473\\
1024	0.940617587368421\\
4096	1.01669013684211\\
8192	1.24587410789474\\
};
\addlegendentry{VSMATCH}

\addplot [color=mycolor2, line width=2.0pt, mark=x, mark options={solid, mycolor2}]
  table[row sep=crcr]{%
1	0.307001928378378\\
4	0.407980089736842\\
16	0.463272527368421\\
64	0.497099931891892\\
256	0.585506258378378\\
1024	0.677180935897436\\
4096	0.738128860243902\\
8192	0.73170738875\\
};
\addlegendentry{VDSVMB}

\addplot [color=mycolor3, line width=2.0pt, mark=triangle, mark options={solid, mycolor3}]
  table[row sep=crcr]{%
1	0.1719653415\\
4	0.326831306666667\\
16	0.39110541525\\
64	0.420306074054054\\
256	0.434918104871795\\
1024	0.465797267317073\\
4096	0.630535908297872\\
8192	0.664523475208333\\
};
\addlegendentry{AS}

\end{axis}

\begin{axis}[%
width=0.25\columnwidth,
height=0.25\columnwidth,
at={(3in,0.481in)},
scale only axis,
xmode=log,
log basis x=2,
xmin=1,
xmax=8192,
xminorticks=true,
xlabel style={font=\color{white!15!black}},
xlabel={$n_p$},
ymin=25,
ymax=100,
axis background/.style={fill=white},
legend style={legend cell align=left, align=left, draw=none,fill=none}
]
\addplot [color=mycolor1, line width=2.0pt, mark=o, mark options={solid, mycolor1}]
  table[row sep=crcr]{%
1	100\\
4	82.8458453099513\\
16	68.047120881722\\
64	65.3770059003052\\
256	59.1273769134134\\
1024	55.9884752097399\\
4096	51.7992085925006\\
8192	42.2705184564858\\
};
\addlegendentry{VSMATCH}

\addplot [color=mycolor2, line width=2.0pt, mark=x, mark options={solid, mycolor2}]
  table[row sep=crcr]{%
1	100\\
4	75.2492428187863\\
16	66.2681057567294\\
64	61.7585939330091\\
256	52.4335861462271\\
1024	45.3352881193448\\
4096	41.5919150318733\\
8192	41.9569261016812\\
};
\addlegendentry{VDSVMB}

\addplot [color=mycolor3, line width=2.0pt, mark=triangle, mark options={solid, mycolor3}]
  table[row sep=crcr]{%
1	100\\
4	52.615933049335\\
16	43.9690515126408\\
64	40.9143127153295\\
256	39.5397063432648\\
1024	36.9184951407071\\
4096	27.2728863236702\\
8192	25.8779934668354\\
};
\addlegendentry{AS}

\end{axis}
\end{tikzpicture}%
}
	\caption{Weak scaling. Average operator complexity of AMG preconditioners and average time $T (s)$ for linear solver.}
	
\end{figure}
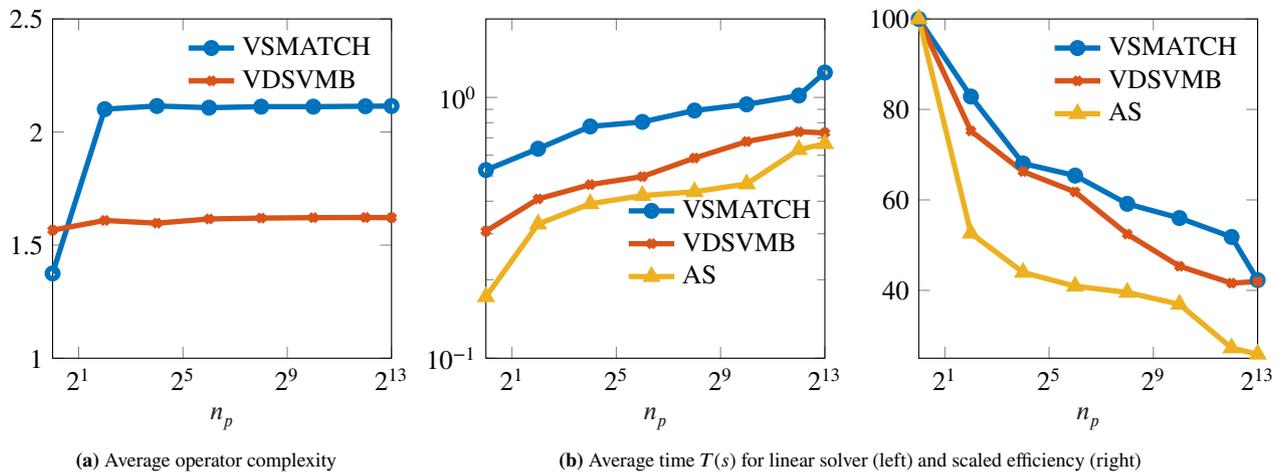
Looking also at the overall behavior of the solving times is useful. If we compare the results in Figure~\ref{fig:linearsystem-scaling} with the ones in Figure~\ref{fig:total-scaledeff}, we observe that the efficiency for the linear system solve phase reflects the global efficiency and is not hampered by the fact that we are building Jacobians and auxiliary matrices, see again Figure~\ref{fig:matrix-build-scaling}, or by the way in which we are performing the function evaluations.

\end{document}